\documentclass[11pt]{amsart}
\usepackage{amsmath,amsfonts,amsthm}
\usepackage{amssymb,latexsym}
\usepackage[colorlinks=true, linkcolor=blue,urlcolor=blue]{hyperref}
\usepackage{graphicx}
\usepackage[all]{xy}
\usepackage{tikz}
\usepackage{layout}

\theoremstyle{plain}
\newtheorem{theorem}{Theorem}[section]
\newtheorem{lemma}[theorem]{Lemma}
\newtheorem{proposition}[theorem]{Proposition}
\newtheorem{corollary}[theorem]{Corollary}

\theoremstyle{definition}

\newtheorem{definition}[theorem]{Definition}
\newtheorem{example}[theorem]{Example}

\theoremstyle{remark}
\newtheorem*{notation}{Notation}
\numberwithin{equation}{section}

\let\emptyset\varnothing
\title{The Category of Ordered Bratteli Diagrams}

\author{Massoud Amini,  George A. Elliott, and Nasser Golestani}

\address[\textbf{Massoud Amini}]{Department of Pure Mathematics\\ Faculty of Mathematical Sciences\\
  Tarbiat Modares University\\
 Tehran\\ Iran}
 \address{School of Mathematics, Institute for Research in Fundamental Sciences (IPM),
Tehran 19395-5746, Iran}
\email {mamini@modares.ac.ir, mamini@ipm.ir}

\address[\textbf{George A. Elliott}]{Department of Mathematics\\ University of Toronto\\ Toronto, Ontario, Canada\ \ M5S 2E4}
\email {elliott@math.toronto.edu}

\address[\textbf{Nasser Golestani}]{Department of Pure Mathematics\\ Faculty of Mathematical Sciences\\
  Tarbiat Modares University\\
 Tehran\\ Iran}
\email {n.golestani@modares.ac.ir}


\begin{document}
\begin{abstract}
A category structure for ordered Bratteli diagrams is proposed in which
isomorphism coincides with the notion
of equivalence of Herman, Putnam, and Skau.
It is shown that the natural one-to-one correspondence between
the category of Cantor minimal  systems
and the category of simple properly ordered Bratteli diagrams
is in fact an equivalence of categories.
This gives   a \emph{Bratteli--Vershik model} for  factor maps  between  Cantor minimal systems.
We give a construction of factor maps between  Cantor
minimal systems in terms of suitable maps (called premorphisms) between
the corresponding ordered Bratteli diagrams,
and we show that every factor map between two Cantor minimal systems is obtained in this way.
Moreover, solving a natural question, we are able to characterize
Glasner and Weiss's notion of weak orbit equivalence
of Cantor minimal systems  in terms of the corresponding
 C*-algebra crossed products.
\end{abstract}
\maketitle
\vspace{-.5cm}
\section{Introduction}\label{secintro}

In 1972, Bratteli  introduced  what are now called Bratteli diagrams
to study AF~algebras \cite{br72}.
He  associated to each AF algebra an infinite
directed graph (see Definition~\ref{defbd2})
and used these very ef{f}ectively to study (and classify) AF~algebras.
Some attributes of an AF algebra (such as its
 ideal structure) can be read of{f} directly from its Bratteli diagram.

 The second author introduced
 the notion of dimension group and gave a
  classification of AF~algebras using K-theory in 1976 \cite{el76}, showing  that the functor
 $\mathrm{K}_{0}: \mathbf{AF}\to \mathbf{DG}$, from the category of
 AF~algebras with  $*$-homomorphisms
 to the category of scaled dimension groups with order-preserving homomorphisms,
is a strong classification functor
 (see also \cite[Sections 5.1--5.3]{el10}).

Recall that a functor $F:\mathcal{C}\to \mathcal{D}$
is called  a {classification functor} (\cite{el10}) if
$F(a)\cong F(b)$ implies $a\cong b$, for each $a,b\in \mathcal{C}$,
and a {strong classification functor} if each isomorphism from $F(a)$ to $F(b)$
is the image of an isomorphism from $a$ to $b$.

In \cite{aeg14}, the authors introduced  the category $\mathbf{BD}$ of  Bratteli diagrams,
isomorphisms of which coincide
with the  notion of equivalence of Bratteli diagrams introduced
by Bratteli, to capture isomorphism of the corresponding AF~algebras.
We showed that the map $\mathcal{B}:\mathbf{AF}\to \mathbf{BD}$,
defined by Bratteli  in \cite{br72} on objects, is in fact a functor. The fact that this is a
strong classification functor \cite[Theorem~3.11]{aeg14} is a
functorial formulation of Bratteli's classification
of AF~algebras in terms of diagrams, and completes his work
from the classification functor point of view
of \cite{el10}.

Bratteli diagrams have been used to study certain dynamical systems.
In 1981, Vershik used Bratteli diagrams to construct the so-called adic
transformations \cite{ve81a, ve81b}.
Based on his work (and the work of Power \cite{po91}),
Herman, Putnam, and Skau introduced the notion of  ordered Bratteli diagram,
and associated a dynamical system to a  properly ordered Bratteli diagram  \cite{hps92}.
They showed that there is a one-to-one correspondence between
  properly ordered Bratteli diagrams
and essentially minimal totally disconnected
dynamical systems \cite[Theorem~4.7]{hps92}. In particular, each Cantor minimal system has a
Bratteli--Vershik model.
This correspondence was used effectively to study
Cantor minimal  systems and in particular to characterize what they called
strong orbit equivalence in terms of isomorphism of dimension groups and the corresponding C*-algebra
crossed products
 \cite{hps92, gps95, gw95, ma02}.
(Simple orbit equivalence is also characterized in \cite{li05}.
In the present work we  do this for weak orbit equivalence.)

Most of the classi{f}ication results concerning Cantor minimal systems
and their associated ordered Bratteli diagrams and ordered K-groups,
obtained up to now, only deal with   isomorphism classes (see
\cite{gps95,gw95,hps92}).
For instance, in \cite{hps92} Herman, Putnam, and Skau, among other things,
showed that two
Cantor minimal systems $(X,\varphi)$ and $(Y,\psi)$ are
conjugate if and only if their associated ordered Bratteli diagrams
are equivalent. An obvious question is then whether one can realize
factor maps (an important notion
in dynamics) from $(X,\varphi)$ to $(Y,\psi)$ in terms of maps
between the
ordered Bratteli diagrams. In particular, one could ask
if $(Y,\psi)$'s being a factor of $(X,\varphi)$ could be decided by
looking at the corresponding ordered Bratteli diagrams.
Sugisaki in \cite{su11} and Host and Glasner in \cite{gh13}  studied
certain factor maps (for instance, almost one-to-one extensions)
 in terms of dimension groups (see also \cite{gps01} and \cite{DD02}).
The functorial classification approach of \cite{el10} (finding
classification functors---which are possibly full or faithful)
is relevant to this question, as it takes general morphisms into account, and could lead,
at least in certain cases, to a classification of morphisms.
This is the main objective of the current paper.
In particular, we obtain a functor $\mathcal{P}$ from the category of
Cantor minimal systems to the category of ordered Bratteli diagrams
and its inverse functor $\mathcal{V}$, leading to a \emph{model} for factor maps
between Cantor minimal systems. Having a model may have many applications. For instance,
the classical result on the existence of the maximal rational equicontinuous factor for
Cantor minimal systems (Theorem~\ref{thm_odo}),
and the uniqueness of a factor map onto an odometer (Proposition~\ref{prop_odo_uni})
 follow easily from this model. (The characterization of almost one-to-one
 extensions and the study of finite-to-one factor maps will be given
 in the forthcoming paper \cite{GH18}.)

 There is a close relation between Cantor minimal systems and
 certain C*-algebras. Indeed, to every Cantor minimal system
 $(X,\varphi)$ there is associated a C*-algebra crossed product
 $\mathrm{C}(X)\rtimes_{\varphi}\mathbb{Z}$ with the same
 ordered $\mathrm{K}_{0}$-group as that system \cite{hps92}.
One expects that every (equivalence) relation between two Cantor
minimal systems has characterizations in terms of C*-algebra crossed products.
This has been shown already for strong orbit equivalence
by Giordano, Putnam, and Skau in \cite{gps95}, and for
orbit equivalence by Lin in \cite{li05}.
(See also \cite{gps95} for characterizations of
{f}lip conjugacy and Kakutani (strong) orbit equivalence.)
However, no characterization for the  weak orbit equivalence  of Glasner and Weiss in terms of
C*-algebras was  known. We use the notion of
tracial equivalence in the sense of Lin \cite{li05} to achieve this goal (Theorem~\ref{thmwoec}).

The structure of the paper is as follows.
Following on the ideas of \cite{aeg14},
we first propose a notion of  morphism between ordered Bratteli diagrams and
obtain the category $\mathbf{OBD}$ of ordered Bratteli diagrams
(Section~\ref{secpre}). Isomorphism in this category coincides with
 equivalence in the sense of Herman, Putnam, and Skau.
We show that the correspondence obtained by Herman, Putnam, and Skau in \cite{hps92}
is an equivalence of categories. More precisely, for
 the category $\mathbf{SDS}$ of scaled essentially minimal totally disconnected dynamical systems
(Definition~\ref{defSDS}), which contains the category of Cantor minimal systems, we construct a contravariant functor
$\mathcal{P}: \mathbf{SDS}\to\mathbf{OBD}$ (Section~\ref{secp}), leading to what might be
viewed as a model for essentially minimal totally disconnected dynamical
systems and their morphisms. In particular,
the functor $\mathcal{P}$ gives  a Bratteli--Vershik model for factor maps
between Cantor minimal systems, which we then use  in the study of factors  of such systems.

In Section~\ref{secp}, we show that the contravariant functor
$\mathcal{P}: \mathbf{SDS}\to\mathbf{OBD}$
is full  and faithful,
and identify the (essential) range of this functor, as the class  of
  properly ordered Bratteli diagrams
$\mathbf{OBD}_{\mathrm{po}}$. This gives an equivalence of categories
  $\mathcal{P}: \mathbf{SDS}\to\mathbf{OBD}_{\mathrm{po}}$
(Theorem~\ref{threqucat}).
We also construct an  inverse to the functor $\mathcal{P}$, a
 contravariant functor
$\mathcal{V} : \mathbf{OBD}_{\mathrm{po}}\to \mathbf{SDS}$, which, naturally,  is
also an equivalence of categories. This latter functor
 gives us a handle on factor maps between Cantor minimal systems,
by graphically constructing certain arrows (premorphisms) between the associated
ordered Bratteli diagrams.
This is in particular useful when one   applies these
functors to morphisms. In this way, one obtains a functorial formulation
(including general morphisms) of the correspondence of \cite{hps92}
between   properly ordered Bratteli diagrams and
essentially minimal totally disconnected dynamical systems (Theorem~\ref{theequcat2}).

In Section~\ref{seccms},
we apply the results of Section~\ref{secp} to certain subcategories of
$\mathbf{SDS}$.
In particular, we show that
the category of Cantor minimal  systems
 is equivalent to the category of
 what Durand, Host, and Skau called
properly ordered Bratteli diagrams; see \cite{dhs99}  (Corollary~\ref{cormdsc} below).
In Subsection~\ref{subseccms}, we focus on
factors of Cantor minimal systems, to illustrate the use of our functorial machinery.
We give concrete examples of the construction of factor maps
using  premorphisms. In particular,
we reprove---by the technique of premorphisms---the fact that every Cantor minimal
system has a maximal odometer factor (Theorem~\ref{thm_odo}).
In \cite{GH18}  more applications of this technique are given.
Indeed the notion of (ordered) premorphism enables us to
construct desired factor maps by using an explicit graphical method.

In   Section~\ref{secwoe}, we give an equivalent condition---in
terms of the  corresponding C*-algebra crossed products---for the weak orbit equivalence of Glasner and Weiss.

\section{The Category of Ordered Bratteli Diagrams}\label{secpre}

In this section we propose a notion of morphism for the category $\mathbf{OBD}$ of ordered Bratteli diagrams.
This construction is  similar to the construction of the category of Bratteli diagrams,
$\mathbf{BD} $,  given in \cite{aeg14}.
In particular, first we need a notion of
(ordered) premorphism.
 We shall see that  isomorphism in this category coincides with  equivalence of
ordered Bratteli diagrams, as defined by Herman, Putnam, and Skau
in \cite{hps92}.

 Let us first recall  and  fix some notation concerning
 Bratteli diagrams. See \cite{BK16,
 BKM09, Du10, dhs99, hps92, gps95, aeg14}
 for more information about
 (simple and non-simple) Bratteli
 diagrams.

  \begin{definition}\label{defbd2}
A \emph{Bratteli diagram} consists of a vertex set $V$ and an edge set $E$ satisfying
the following conditions. We have a decomposition of $V$ as a disjoint union
$V_{0}\cup V_{1}\cup \cdots$, where each $V_{n}$ is finite and non-empty
and $V_{0}$ has exactly one element, $v_{0}$. Similarly, $E$ decomposes as
a disjoint union $E_{1}\cup E_{2}\cup \cdots$, where each $E_{n}$ is finite
and non-empty. Moreover, we have maps $r,s:E\to V$ such that $r(E_{n})\subseteq V_{n}$
and $s(E_{n})\subseteq V_{n-1}$, $n=1,2,3,\ldots$ ($r=$ range, $s=$ source).
We also assume that $s^{-1}\{v\}$ is non-empty for all $v$ in $V$ and
$r^{-1}\{v\}$ is non-empty for all $v$ in $V\setminus V_{0}$.
 Let us denote such a $B$ by the  diagram

\[
\xymatrix{V_{0}\ar[r]^-{\resizebox{1.15em}{.65em}{$E_{1}$}}
 &V_{1}\ar[r]^-{\resizebox{1.15em}{.65em}{$E_{2}$}}
 &V_{2}\ar[r]^-{\resizebox{1.15em}{.65em}{$E_{3}$}}
 &\cdots\ .
 }
\]
\end{definition}

In the preceding definition, if we fix a total order on
each $V_{n}$, then to each edge set $E_{n}$   a matrix  $\mathrm{M}(E_{n})$ is associated,
called the  \emph{multiplicity matrix} of $E_{n}$ (also  called the ``incidence matrix''  \cite{gps95}).

Let $k,l$ be integers with $0\leq k< l$. Let $E_{k,l}=E_{k+1}\circ E_{k}\circ\cdots\circ E_{l}$
denote the set of all paths from $V_{k}$ to $V_{l}$, that is, the tuples $(e_{k+1},\ldots,e_{l})$ where $ e_{i}\in E_{i},$ for $ i=k+1,\ldots,l,$ with $r(e_{i})=s(e_{i+1})$,
 for $i=k+1,\ldots,l-1$.
In particular,  $E_{k,k}=\{(v,v)\mid v\in V_{k}\}$  is an edge set from $V_{k}$ to itself
We identify $E_{k,k}$ with its multiplicity matrix.

\begin{definition}[\cite{dhs99, BK16, BKM09}]\label{defsbd}
Let  $B=(V,E)$ be a  Bratteli diagram (as in Definition~\ref{defbd2}).
$B$ is called \textit{simple} if there exists a telescoping
$(V',E')$ of $(V,E)$ such that the multiplicity matrices of $(V',E')$
have only non-zero elements at each level. In other words, $B$ is simple if for each
$n\geq 0$ there is $m>n$ such that, for every   $v\in V_{n}$ and  every
$w\in V_{m}$, there is a path in $E_{n,m}$ from $v$ to $w$.
\end{definition}

\begin{definition}\label{defobd}
An \textit{ordered Bratteli diagram} is a Bratteli diagram $(V,E)$ as in
Definition~\ref{defbd2} together with an order relation
$\geq$ on $E$ such that $e,e' \in E$ are comparable if, and only, if $r(e)=r(e')$. In other words,
we have a linear order on each set $r^{-1}\{v\}$, for every $v\in V\setminus V_{0}$.
\end{definition}

If $(V,E,\geq )$ is an ordered Bratteli diagram and $k , l$ are integers with $0\leq k < l$,
 then the set
 $E_{k,l}$
 may be given an induced (lexicographic)
order  \cite{dhs99, hps92}.

For an ordered Bratteli diagram
$(V,E,\geq)$, denote by $E_{\max}$ and  $E_{\min}$
the set of maximal and minimal edges of $E$, respectively.

\begin{definition}\label{defpobd}
Let  $B=(V,E,\geq)$ be an ordered Bratteli diagram. We say that $B$  is  \textit{properly
ordered}
if there are unique infinitely long paths
in $E_{\max}$ and  $E_{\min}$, that is, there is only one sequence
$(e_{1},e_{2},\ldots)$	with each $e_{i}$ in $E_{\max}$
 and $s(e_{i+1})=r(e_{i})$,
for all $i\geq 1$, and the same holds for $E_{\min}$.
\end{definition}

Note that, properly ordered Bratteli diagrams 
(in the sense of the preceding definition) are called
essentially simple in \cite{hps92, dhs99, Du10}.
We use the now standard term
``properly ordered" (see, e.g., \cite{BK16}).

Let us define the category of ordered Bratteli diagrams.
We need a notion of (ordered)
 premorphism before considering the final
notion of morphism.
Denote by $\mathbf{OBD}$ the class of all ordered Bratteli diagrams.

\begin{definition}\label{defpreobd}
Let $B=(V,E,\geq)$ and $C=(W,S,\geq)$
be  ordered Bratteli diagrams.
By an \emph{ordered} \emph{premorphism}
(or just a \emph{premorphism}
if there is no confusion) $f: B\to C$ we mean
 a triple $(F, (f_{n})_{n=0}^{\infty},\geq )$ where
$(f_{n})_{n=0}^{\infty}$ is a cofinal
(i.e., unbounded)
 sequence of  positive integers with
$f_{0}=0\leq f_{1}\leq f_{2}\leq\cdots$,
 $F$ consists of a disjoint union
 $F_{0}\cup F_{1}\cup F_{2}\cup\cdots$ together with a pair of range and source maps
 $r:F\to W$, $s:F\to V$, and
  $\geq$ is a partial order on $F$ such that:

\begin{enumerate}
\item\label{defprebd_it0}
 each $F_{n}$ is a non-empty finite set, $s(F_{n})\subseteq V_{n}$,
 $r(F_{n})\subseteq W_{f_{n}}$,  $F_{0}$
 is a singleton,\\
\hspace*{-1.4cm} $s^{-1}\{v\}$ is non-empty for all $v$ in $V$, and
$r^{-1}\{w\}$ is non-empty for all $w$ in $W$;
\item\label{defpreobd_it1}
$e,e'\in F$ are comparable if and only if $r(e)=r(e')$, and $\geq$ is a linear order
on\\
\hspace*{-1.4cm} $r^{-1}\{w\}$, for
all $w\in W$;

\item\label{defpreobd_it2}
 the diagram of $f:B\to C$,

  \[
\xymatrix{V_{0}\ar[r]^{E_{1}}\ar[d]_{F_{0}}
 &V_{1}\ar[r]^-{E_{2}}\ar[d]_{F_{1}} &V_{2}\ar[r]^-{E_{3}}\ar[d]_{F_{2}} &\cdots \ \ \  \\
 W_{f_{0}}\ar[r]_{S_{f_{0},f_{1}}}
 &W_{f_{1}}\ar[r]_{S_{f_{1},f_{2}}} &W_{f_{2}}\ar[r]_{S_{f_{2},f_{3}}}&\cdots \   ,
 }
\]
\hspace*{-1.4cm} commutes.
The (ordered) commutativity of the diagram of $f$  means that
for each\\
\hspace*{-1.4cm} $n\geq 0$,
 $E_{n+1}\circ F_{n+1}\cong F_{n}\circ S_{f_{n},f_{n+1}}$, i.e.,
 there is a (necessarily unique) bijective map\\
\hspace*{-1.4cm} from $E_{n+1}\circ F_{n+1}$ to
 $F_{n}\circ S_{f_{n},f_{n+1}}$
 preserving the order and intertwining the respective\\
\hspace*{-1.4cm} source and range
maps.
\end{enumerate}
\end{definition}

If $B$ and $C$ in the preceding definition
are  (unordered) Bratteli diagrams then
a pair $f=(F, (f_{n})_{n=0}^{\infty})$
with the properties stated in the preceding definition (without Condition~\eqref{defpreobd_it1})
 is called a \emph{premorphism}
from $B$ to $C$. Note that, in this case,
we require only (unordered) commutativity
of the diagram of $f$,
that is, for each $n\geq 0$, each $v\in V_{n}$, and each $w\in W_{f_{n+1}}$,
 the number
  of paths from $v$ to $w$ passing through $W_{f_{n}}$ and the
 number through $V_{n+1}$ are equal.
 This is equivalent to saying that
  for any positive integer $n$,
    $\mathrm{M}(F_{n+1})\mathrm{M}(E_{n+1})=\mathrm{M}(S_{f_{n},f_{n+1}})\mathrm{M}(F_{n})$.

We remark that the ordered commutativity required in Definition~\ref{defpreobd}
is essential. In fact, if $f$ is a premorphism  (i.e., only unordered commutativity holds), then one obtains a continuous map
between the associated Bratteli compacta.
However, if $f$ is an ordered premorphism  (i.e.,  ordered commutativity holds),
then one gets not only a continuous map but also a homomorphism between
the associated dynamical systems (see Subsection~\ref{subsecv}).
See Figure~\ref{fig_oc} for an illustrative example of  ordered and unordered commutativity.

\begin{figure}
\begin{center}
\begin{tikzpicture}[scale=1.2]
\filldraw (1.2,1.5) circle [radius=0.08];
\filldraw (1.8,1.5) circle [radius=0.08];
\filldraw (1.5,0) circle [radius=0.08];

\filldraw (3.5,1.5) circle [radius=0.08];
\filldraw (3.5,0) circle [radius=0.08];
\draw[->, thick] (1.27,1.6) to [out=30,in=150] (3.38,1.63);
\draw[->, thick] (1.93,1.5) to [out=10,in=170] (3.38,1.5);
\draw[->, thick] (1.65,0) to [out=10,in=170] (3.38,0);
\draw[->, thick] (3.5,1.38) to (3.5,0.13);
\draw[->, thick] (1.22,1.38) to (1.43,0.1);
\draw[->, thick] (1.79,1.38) to (1.57,0.1);

\node at (1.27,0.3) {\tiny{1}};
\node at (1.72,0.3) {\tiny{2}};

\node at (3.2,1.86) {\tiny{2}};
\node at (3.2,1.35) {\tiny{1}};

\filldraw (6.2,1.5) circle [radius=0.08];
\filldraw (6.8,1.5) circle [radius=0.08];
\filldraw (6.5,0) circle [radius=0.08];

\filldraw (8.5,1.5) circle [radius=0.08];
\filldraw (8.5,0) circle [radius=0.08];
\draw[->, thick] (6.27,1.6) to [out=30,in=150] (8.38,1.63);
\draw[->, thick] (6.93,1.5) to [out=10,in=170] (8.38,1.5);
\draw[->, thick] (6.65,0) to [out=10,in=170] (8.38,0);
\draw[->, thick] (8.5,1.38) to (8.5,0.13);
\draw[->, thick] (6.22,1.38) to (6.43,0.1);
\draw[->, thick] (6.79,1.38) to (6.57,0.1);

\node at (6.27,0.3) {\tiny{1}};
\node at (6.72,0.3) {\tiny{2}};

\node at (8.2,1.86) {\tiny{1}};
\node at (8.2,1.35) {\tiny{2}};

\node at (1.19,1.68) {\tiny{$u$}};
\node at (1.8,1.68) {\tiny{$v$}};
\node at (3.5,-.2) {\tiny{$w$}};
\node at (3.5,1.68) {\tiny{$y$}};
\node at (1.5,-.2) {\tiny{$z$}};

\node at (6.19,1.68) {\tiny{$u$}};
\node at (6.8,1.68) {\tiny{$v$}};
\node at (8.5,-.2) {\tiny{$w$}};
\node at (8.5,1.68) {\tiny{$y$}};
\node at (6.5,-.2) {\tiny{$z$}};
\end{tikzpicture}
\end{center}
\caption{The di{f}ference between \emph{unordered} commutativity (left diagram)
and \emph{ordered} commutativity (right diagram). In
both diagrams the number of paths from  $u$ to
$w$ passing through $y$ and the number through $z$
are equal (which is one here), and  the same for paths from $v$ to $w$. However,
in the left diagram the source map is not preserved, since
the first path ending in $w$ and passing through $y$ starts
at $v$ while the first path ending in $w$ and passing through $z$ starts
at $u$. In the right diagram the source map is preserved.
}\label{fig_oc}
\end{figure}

We give an illustrative example of an ordered premorphism  in
Figure~\ref{fig_Toe}. We use thick and
curved arrows to depict the edges of premorphisms.

\begin{example}\label{exodtoe}
In Figure~\ref{fig_Toe}, a premorphism $f:B\to C$ is depicted where
$B$ is the odometer of type
$(k_{n})_{n=1}^{\infty}$ with $k_{1}=1$ and $k_{n}=3$ for $n\geq 2$
(see Definition~\ref{defodo} below),
and $C$ (with left-to-right order) is a Toeplitz system.
The (ordered)
commutativity needed in Definition~\ref{defpreobd}
can be checked easily at each level. Note that since $B$ has only one vertex at each level,
 ordered commutativity (as in Definition~\ref{defpreobd}) is the same as
commutativity   for $f$.
As we will see in Subsection~\ref{subsecv}, applying the
functor $\mathcal{V}$,
 we get a factor map $\mathcal{V}([f]): \mathcal{V}(C) \to \mathcal{V}(B)$
as defined before Lemma~\ref{lemtech}. In fact, $\mathcal{V}(B)$ is the maximal rational
equicontinuous factor of $\mathcal{V}(C)$ (see Theorem~\ref{thm_odo} below,
and \cite{GJ00}).
\end{example}

\begin{figure}
\begin{center}
\begin{tikzpicture}[scale=1.1]

\filldraw (3,10) circle [radius=0.1];
\filldraw (3,7) circle [radius=0.1];
\filldraw (3,4) circle [radius=0.1];
\filldraw (3,1) circle [radius=0.1];

\draw (3,9.87)--(3,7.14);

\draw (2.9,6.87)--(2.9,4.14);
\draw (3,6.87)--(3,4.14);
\draw (3.1,6.87)--(3.1,4.14);

\draw (2.9,3.87)--(2.9,1.14);
\draw (3,3.87)--(3,1.14);
\draw (3.1,3.87)--(3.1,1.14);

\node at (2.8,4.45) {\tiny{1}};
\node at (3,4.45) {\tiny{2}};
\node at (3.2,4.45) {\tiny{3}};

\node at (2.8,1.45) {\tiny{1}};
\node at (3,1.45) {\tiny{2}};
\node at (3.2,1.45) {\tiny{3}};

\filldraw (7,10) circle [radius=0.1];
 \filldraw (6,7) circle [radius=0.1];
\filldraw (8,7) circle [radius=0.1];
 \filldraw (6,4) circle [radius=0.1];
\filldraw (8,4) circle [radius=0.1];
 \filldraw (6,1) circle [radius=0.1];
\filldraw (8,1) circle [radius=0.1];


\draw (6.94,9.9)--(6.01,7.13);
\draw (7.085,9.915)--(7.97,7.12);

\draw (5.97,6.87)--(5.97,4.14);
\draw (5.97,3.87)--(5.97,1.14);
\draw (6.06,6.87)--(6.06,4.14);
\draw (6.06,3.87)--(6.06,1.14);

\draw (7.98,6.87)--(7.98,4.14);
\draw (7.98,3.87)--(7.98,1.14);
\draw (8.06,6.87)--(8.06,4.14);
\draw (8.06,3.87)--(8.06,1.14);

\draw (7.9,6.9)--(6.08,4.1);
\draw (7.9,3.9)--(6.08,1.1);
\draw (6.1,6.9)--(7.95,4.15);
\draw (6.1,3.9)--(7.95,1.15);

\node at (8.15,4.45) {\tiny{3}};
\node at (7.9,4.45) {\tiny{2}};
\node at (7.65,4.45) {\tiny{1}};
\node at (6.14,4.45) {\tiny{2}};
\node at (5.91,4.45) {\tiny{1}};
\node at (6.42,4.45) {\tiny{3}};
\node at (8.15,1.45) {\tiny{3}};
\node at (7.9,1.45) {\tiny{2}};
\node at (7.65,1.45) {\tiny{1}};
\node at (6.14,1.45) {\tiny{2}};
\node at (5.91,1.45) {\tiny{1}};
\node at (6.42,1.45) {\tiny{3}};

\draw[->, thick] (3.1,10.1) [out=10,in=170] to (6.9,10.1);

\draw[->, thick] (3.1,7.1) [out=10,in=170] to (5.9,7.1);
\draw[->, thick] (3.1,7.15) [out=30,in=150] to (7.9,7.1);

\draw[->, thick] (3.15,4.1) [out=10,in=170] to (5.9,4.1);
\draw[->, thick] (3.15,4.15) [out=30,in=150] to (7.9,4.1);

\draw[->, thick] (3.15,1.1) [out=10,in=170] to (5.9,1.1);
\draw[->, thick] (3.15,1.15) [out=30,in=150] to (7.9,1.1);
\node[very thick] at (3,0.2) {\vdots};
\node[very thick] at (7,0.2) {\vdots};
\node at (3,10.7) {$B$};
\node at (7,10.7) {$C$};
\node at (5,10.9) {$f$};
\draw[->] (4,10.7) to (6,10.7);
\end{tikzpicture}
\end{center}
\caption{An  ordered premorphism $f$
from the ordered Bratteli diagram of an odometer $B$ to  that of a
Toeplitz system $C$ (see Example~\ref{exodtoe}).}\label{fig_Toe}
\end{figure}

In a way similar to \cite{aeg14},
we define an isomorphism relation on the class of ordered premorphisms
and we define the composition of two ordered premorphisms.

\begin{definition}\label{defisopre}
Let $B,C\in \mathbf{OBD}$ and let
$f,f':B\to C$ be a pair of ordered premorphisms where
$f=(F,(f_{n})_{n=0}^{\infty},\geq)$ and
$f'=(F',(f'_{n})_{n=0}^{\infty}, \geq')$.
We shall say that $f$ is \textit{isomorphic} to $f'$, and write
$f\cong f'$, if $f_{n}=f'_{n}$, $n\geq 0$, and there is a bijective
map from $F$ to $F'$, preserving the order
and the range and source maps.
This is an equivalence relation on the class of ordered
premorphisms from $B$ to $C$.
We denote the equivalence class of $f$  by
$\overline{f}$.
Let $B$, $C$, and $D$ be objects in $\mathbf{OBD}$ and
let $f:B\to C$ and $g:C\to D$ be
ordered premorphisms;
$f=(F,(f_{n})_{n=0}^{\infty}, \geq)$,
$g=(G,(g_{n})_{n=0}^{\infty}, \geq)$,
where
$F=\bigcup_{n=0}^{\infty}F_{n}$ and
$G=\bigcup_{n=0}^{\infty}G_{n}$ (disjoint unions).
The \emph{composition} of $f$ and $g$ is defined as
$gf=(H,(h_{n})_{n=0}^{\infty},\geq)$, where $h_{n}=g_{f_{n}}$,
$H=\bigcup_{n=0}^{\infty}H_{n}$, and
$H_{n}=F_{n}\circ G_{f_{n}}$, $n\geq 0$ (i.e., the set of all paths from $s(F_{n})$ to
$r(G_{f_{n}})$. The partial order $\geq$
on $H$ is the induced lexicographic order.
Also, set
$\overline{g}\overline{f}=\overline{gf}$.
\end{definition}

It is not hard to see that
the class $\mathbf{OBD}$, with  ordered  premorphisms
modulo the relation of isomorphism (see above)
is a category.
We shall refer to this as the category of ordered Bratteli diagrams
with ordered premorphisms.
 Two ordered Bratteli diagrams are isomorphic  in the category $\mathbf{OBD}$ with (ordered)
premorphisms
if, and only if, they are isomorphic  in the sense of Herman, Putman, and Skau
(\cite{hps92}).

We  define an equivalence relation on ordered premorphisms.

\begin{definition}\label{defeq}
Let $B, C$ be ordered Bratteli diagrams
and let $f,g: B\to C$ be ordered premorphisms
with $B=(V,E,\geq)$, $C=(W,S,\geq)$,
$f=(F,(f_{n})_{n=0}^{\infty},\geq)$, and
$g=(G,(g_{n})_{n=0}^{\infty},\geq)$.
We shall say that $f$ is \emph{equivalent} to $g$, and write $f\sim g$,
if there are sequences $(n_{k})_{k=1}^{\infty}$ and $(m_{k})_{k=1}^{\infty}$ of
positive integers such that $n_{k}<m_{k}<n_{k+1}$ and $f_{n_{k}}<g_{m_{k}}<f_{n_{k+1}}$
for each $k\geq 1$, and the diagram

\[
\xymatrix{V_{n_{\resizebox{.28em}{.28em}{1}}}
\ar[r]\ar[d]_{F_{n_{\resizebox{.26em}{.26em}{1}}}}
 & V_{m_{\resizebox{.28em}{.28em}{1}}}
 \ar[r]\ar[d]_{G_{m_{\resizebox{.26em}{.26em}{1}}}}  &
 V_{n_{\resizebox{.28em}{.28em}{2}}}
 \ar[r] \ar[d]_{F_{n_{\resizebox{.26em}{.26em}{2}}}}
 & V_{m_{\resizebox{.28em}{.28em}{2}}}
 \ar[r]\ar[d]_{G_{m_{\resizebox{.26em}{.26em}{2}}}} &\cdots \\
 W_{f_{n_{\resizebox{.26em}{.26em}{1}}}}\ar[r]
 &W_{g_{m_{\resizebox{.26em}{.26em}{1}}}}\ar[r] &
 W_{f_{n_{\resizebox{.26em}{.26em}{2}}}}\ar[r]
 & W_{g_{m_{\resizebox{.26em}{.26em}{2}}}}\ar[r]&\cdots
 }
\]

\noindent is (ordered) commutative, i.e., each minimal square commutes: for each $k\geq 1$,
\begin{equation*}
E_{n_{k},m_{k}}\circ G_{m_{k}}\cong F_{n_{k}}\circ S_{f_{n_{k}},g_{m_{k}}},
\end{equation*}
\begin{equation*}
   E_{m_{k},n_{k+1}}\circ F_{n_{k+1}}\cong
G_{m_{k}}\circ S_{g_{m_{k}},f_{n_{k+1}}}.
\end{equation*}
\end{definition}

It is easily checked that $\sim$ is an equivalence relation on the class of
 ordered premorphisms from $B$ to $C$.
Let us call the equivalence classes  \textit{ordered morphisms},
or if there is no confusion, just \textit{morphisms}, in  $\mathbf{OBD}$.	
We shall denote the equivalence class of an
ordered premorphism
$f:B\to C$  by $[f]:B\to C$, or, if there is  no confusion, just   by $f$.

The composition of morphisms $[f]:B\to C$ and $[g]:C\to D$ is  defined as
$[gf]:B\to D$ where $gf$ is the composition of ordered premorphisms
(see Definition~\ref{defisopre}).
This composition is
well~defined. The first statement of
the next result is proved in a way
similar to the proof
 of \cite[Theorem~2.7]{aeg14}.
 The second statement is easy to prove.

\begin{proposition}\label{thrcatobd}
The class $\mathbf{OBD}$, with
(ordered) morphisms as defined above, is a category. Two ordered Bratteli diagrams are isomorphic  in this category
if and only if they are equivalent  in the sense of Herman, Putnam, and Skau.
\end{proposition}

Let us refer to the category $\mathbf{OBD}$ with (ordered)
morphisms as defined above as the
category of ordered Bratteli diagrams.

We shall now give  two  alternative formulations of the definition of equivalence for premorphisms
(Definition~\ref{defeq}).
The first one will be  used in a number of places later.

\begin{definition}\label{defeq2}
Let $f,g: B\to C$ be ordered premorphisms in
$\mathbf{OBD}$, with $B=(V,E,\geq)$, $C=(W,S,\geq)$,
$f=(F,(f_{n})_{n=0}^{\infty},\geq)$, and
$g=(G,(g_{n})_{n=0}^{\infty},\geq)$.
We shall say that $f$ is \emph{equivalent} to $g$, in the second sense,
 if for each $n\geq 0$ there is an $m\geq f_{n},g_{n}$ such that
 $F_{n}\circ S_{f_{n},m}\cong G_{n}\circ S_{g_{n},m}$, and \emph{equivalent} to $g$, in the third sense,
 if for each $n\geq 0$ and for each $k\geq n$, there is an
 $m\geq f_{n},g_{k}$ such that $F_{n}\circ S_{f_{n},m}\cong E_{n,k}\circ G_{k}\circ S_{g_{k},m}$.
\end{definition}

Using an analogue of \cite[Proposition~2.11]{aeg14}, one can see that
Definitions~\ref{defeq} and  \ref{defeq2}  are
 equivalent.

It might be noted that the category of Bratteli diagrams could also be
described in terms of the general category construction of inductive limits starting from
single-step Bratteli diagrams (see, e.g., \cite{gro}).

\begin{figure}
\begin{center}
\begin{tikzpicture}[scale=1.5]

\filldraw (2,10) circle [radius=0.1];
\filldraw (0,7) circle [radius=0.1];
\filldraw (2,7) circle [radius=0.1];
\filldraw (4,7) circle [radius=0.1];
\filldraw (0,4) circle [radius=0.1];
\filldraw (2,4) circle [radius=0.1];
\filldraw (4,4) circle [radius=0.1];
\filldraw (0,1) circle [radius=0.1];
\filldraw (2,1) circle [radius=0.1];
\filldraw (4,1) circle [radius=0.1];

\draw (1.9,9.9)--(0,7.12);
\draw (2.1,9.9)--(3.91,7.1);
\draw (2,9.87)--(2,7.14);
\draw (0,6.87)--(0,4.14);
\draw (0,3.87)--(0,1.14);
\draw (2,6.87)--(2,4.14);
\draw (2,3.87)--(2,1.14);
\draw (3.94,6.87)--(3.94,4.14);
\draw (3.94,3.87)--(3.94,1.14);
\draw (4.06,6.87)--(4.06,4.14);
\draw (4.06,3.87)--(4.06,1.14);
\draw (1.9,6.9)--(0.08,4.1);
\draw (1.9,3.9)--(0.08,1.1);
\draw (3.9,6.9)--(2.08,4.1);
\draw (3.9,3.9)--(2.08,1.1);
\draw (2.1,6.9)--(3.91,4.1);
\draw (2.1,3.9)--(3.91,1.1);
\draw (0.1,6.9)--(1.91,4.1);
\draw (0.1,3.9)--(1.91,1.1);
\draw (0.15,6.95)--(3.87,4.05);
\draw (0.15,3.95)--(3.87,1.05);

\node at (-0.08,4.35) {\tiny{1}};
\node at (0.15,4.35) {\tiny{2}};
\node at (1.63,4.35) {\tiny{1}};
\node at (1.93,4.35) {\tiny{2}};
\node at (2.15,4.35) {\tiny{3}};
\node at (3.5,4.23) {\tiny{1}};
\node at (3.66,4.35) {\tiny{2}};
\node at (3.88,4.35) {\tiny{3}};
\node at (4.15,4.35) {\tiny{4}};

\node at (-0.08,1.35) {\tiny{1}};
\node at (0.15,1.35) {\tiny{2}};
\node at (1.63,1.35) {\tiny{1}};
\node at (1.93,1.35) {\tiny{2}};
\node at (2.15,1.35) {\tiny{3}};
\node at (3.5,1.23) {\tiny{1}};
\node at (3.66,1.35) {\tiny{2}};
\node at (3.88,1.35) {\tiny{3}};
\node at (4.15,1.35) {\tiny{4}};

\filldraw (8,10) circle [radius=0.1];
 \filldraw (7,7) circle [radius=0.1];
\filldraw (9,7) circle [radius=0.1];
 \filldraw (7,4) circle [radius=0.1];
\filldraw (9,4) circle [radius=0.1];
 \filldraw (7,1) circle [radius=0.1];
\filldraw (9,1) circle [radius=0.1];

\draw (7.885,9.96)--(6.92,7.1);
\draw (7.91,9.92)--(6.96,7.11);
\draw (7.94,9.9)--(7.01,7.13);
\draw (7.98,9.88)--(7.05,7.13);
\draw (8.025,9.88)--(7.08,7.1);

\draw (8.06,9.88)--(8.93,7.11);
\draw (8.085,9.915)--(8.97,7.12);
\draw (8.115,9.95)--(9.01,7.13);
\draw (8.14,10)--(9.06,7.11);

\draw (6.98,6.87)--(6.98,4.14);
\draw (6.98,3.87)--(6.98,1.14);
\draw (7.06,6.87)--(7.06,4.14);
\draw (7.06,3.87)--(7.06,1.14);

\draw (8.98,6.87)--(8.98,4.14);
\draw (8.98,3.87)--(8.98,1.14);
\draw (9.06,6.87)--(9.06,4.14);
\draw (9.06,3.87)--(9.06,1.14);

\draw (8.9,6.9)--(7.08,4.1);
\draw (8.9,3.9)--(7.08,1.1);
\draw (7.1,6.9)--(8.95,4.15);
\draw (7.1,3.9)--(8.95,1.15);

\node at (9.15,4.7) {\tiny{3}};
\node at (8.9,4.7) {\tiny{2}};
\node at (8.67,4.7) {\tiny{1}};
\node at (7.14,4.7) {\tiny{2}};
\node at (6.92,4.7) {\tiny{1}};
\node at (7.35,4.65) {\tiny{3}};
\node at (9.15,1.7) {\tiny{3}};
\node at (8.9,1.7) {\tiny{2}};
\node at (8.67,1.7) {\tiny{1}};
\node at (7.14,1.7) {\tiny{2}};
\node at (6.92,1.7) {\tiny{1}};
\node at (7.35,1.65) {\tiny{3}};

\draw[->, thick] (2.1,10.1) to [out=10,in=170] (7.9,10.1);
\draw[->, thick] (0.15,7.05) to [out=30,in=125] (6.9,7.15);
\draw[->, thick] (0.13,7.1) to [out=30,in=125] (6.9,7.28);
\draw[->, thick] (2.15,7.1) to [out=30,in=155] (6.85,7.05);
\draw[->, thick] (2.15,7.04) to [out=30,in=155] (6.85,6.95);
\draw[->, thick] (4.14,7.01) to [out=-15,in=190] (6.85,6.85);

\draw[->, thick] (0.16,7.18) to [out=32,in=145] (8.85,7.15);
\draw[->, thick] (2.15,7.15) to [out=32,in=145] (8.85,7.05);
\draw[->, thick] (4.12,6.95) to [out=-20,in=200] (8.86,6.96);
\draw[->, thick] (4.11,6.88) to [out=-20,in=205] (8.84,6.87);
\draw[->, thick] (0.15,4.05) to [out=30,in=125] (6.9,4.15);
\draw[->, thick] (0.13,4.1) to [out=30,in=125] (6.9,4.28);
\draw[->, thick] (2.15,4.1) to [out=30,in=155] (6.85,4.05);
\draw[->, thick] (2.15,4.04) to [out=30,in=155] (6.85,3.95);
\draw[->, thick] (4.14,4.01) to [out=-15,in=190] (6.85,3.85);

\draw[->, thick] (0.16,1.18) to [out=32,in=145] (8.85,1.15);
\draw[->, thick] (2.15,1.15) to [out=32,in=145] (8.85,1.05);
\draw[->, thick] (4.12,0.95) to [out=-20,in=200] (8.86,0.96);
\draw[->, thick] (4.11,0.88) to [out=-20,in=205] (8.84,0.87);
\draw[->, thick] (0.15,1.05) to [out=30,in=125] (6.9,1.15);
\draw[->, thick] (0.13,1.1) to [out=30,in=125] (6.9,1.28);
\draw[->, thick] (2.15,1.1) to [out=30,in=155] (6.85,1.05);
\draw[->, thick] (2.15,1.04) to [out=30,in=155] (6.85,0.95);
\draw[->, thick] (4.14,1.01) to [out=-15,in=190] (6.85,0.85);

\draw[->, thick] (0.16,4.18) to [out=32,in=145] (8.85,4.15);
\draw[->, thick] (2.15,4.15) to [out=32,in=145] (8.85,4.05);
\draw[->, thick] (4.12,3.95) to [out=-20,in=200] (8.86,3.96);
\draw[->, thick] (4.11,3.88) to [out=-20,in=205] (8.84,3.87);

\node at (6.85,7.5) {\tiny{1}};
\node at (6.7,7.3) {\tiny{3}};
\node at (6.52,7.3) {\tiny{2}};
\node at (6.5,7) {\tiny{4}};
\node at (6.7,6.72) {\tiny{5}};

\node at (8.62,7.4) {\tiny{1}};
\node at (8.45,7.2) {\tiny{2}};
\node at (8.6,6.95) {\tiny{3}};
\node at (8.65,6.7) {\tiny{4}};
\node at (6.85,4.5) {\tiny{1}};
\node at (6.7,4.3) {\tiny{3}};
\node at (6.52,4.3) {\tiny{2}};
\node at (6.5,4) {\tiny{4}};
\node at (6.7,3.72) {\tiny{5}};

\node at (8.62,4.4) {\tiny{1}};
\node at (8.45,4.2) {\tiny{2}};
\node at (8.6,3.95) {\tiny{3}};
\node at (8.65,3.7) {\tiny{4}};
\node at (6.85,1.5) {\tiny{1}};
\node at (6.7,1.3) {\tiny{3}};
\node at (6.52,1.3) {\tiny{2}};
\node at (6.5,1) {\tiny{4}};
\node at (6.7,0.72) {\tiny{5}};

\node at (8.62,1.4) {\tiny{1}};
\node at (8.45,1.2) {\tiny{2}};
\node at (8.6,0.95) {\tiny{3}};
\node at (8.65,0.7) {\tiny{4}};

\node[very thick] at (2,0.2) {\vdots};
\node[very thick] at (8,0.2) {\vdots};
\node at (2,10.7) {$B=(V,E,\geq)$};
\node at (8,10.7) {$C'=(W',S',\geq')$};
\node at (5,10.9) {$f=(F, (n)_{n=0}^{\infty},\geq )$};
\draw[->] (3,10.7) to (6.8,10.7);
\end{tikzpicture}
\end{center}
\caption{An ordered premorphism $f$ from $B$ to $C$.
 See Example~\ref{exa_Chacon}.}\label{fig_Chacon}
\end{figure}

We close this section with  another illustrative example of an ordered premorphism $f$.
We will  construct the inverse of the morphism $[f]$ in Example~\ref{exa_Chacon2}.
Thus, $[f]$ is an isomorphism in the category $\mathbf{OBD}$.

\begin{example}\label{exa_Chacon}
Consider  the Chacon substitution system $(X,\varphi)$
 described in \cite{GJ00}, i.e.,
 the substitution minimal system associated to the 
 Chacon substitution $0\to 0010$, $1\to 1$.
Let $C=(W,S,\geq)$ be the Bratteli--Vershik model for $(X,\varphi)$
as explained in \cite[Section~4.2]{GJ00}. The diagram $C$ is
drawn on the left in  Figure~\ref{fig_Chacon2}, below.
Let $C'=(W',S',\geq')$ be the  telescoping  to the sequence $0,2,3,4,\ldots$ of  $C$.
The diagram $C'$ is drawn on the right in Figure~\ref{fig_Chacon}.
Let $B=(V,E,\geq)$ be the properly ordered Bratteli diagram drawn on the left in Figure~\ref{fig_Chacon}.
One can check that $f:B\to C$ in Figure~\ref{fig_Chacon} is an ordered premorphism.
In the notation of Definition~\ref{defpreobd}, the multiplicity matrices
are the following:
\[
\mathrm{M}(E_{1})
=\left(
 \begin{smallmatrix}
 1  \\
 1 \\
 1
 \end{smallmatrix}
 \right), \
\mathrm{M}(E_{n})
=\left(
 \begin{smallmatrix}
 1 & 1 & 0 \\
 1 & 1 & 1\\
 1 & 1 & 2
 \end{smallmatrix}
 \right), \
 \mathrm{M}(S_{1}')
=\left(
 \begin{smallmatrix}
 5  \\
 4
 \end{smallmatrix}
 \right), \
\mathrm{M}(S_{n}')
=\left(
 \begin{smallmatrix}
 2 & 1  \\
 1 & 2
 \end{smallmatrix}
 \right),\  \text{for}\ n\geq 2;
 \]
 \[
  \mathrm{M}(F_{0})=(1),\
 \mathrm{M}(F_{n})
=\left(
 \begin{smallmatrix}
 2 & 2 & 1 \\
 1 & 1 & 2
 \end{smallmatrix}
 \right),
 \  \text{for}\ n\geq 1.
\]
In fact, these two diagrams are equivalent (see \cite[Section~4.2]{GJ00}).
To show this, we will construct the inverse of the morphism $[f]$ in Example~\ref{exa_Chacon2}.
\end{example}

\section{Functors between  Categories of Dynamical Systems and Bratteli Diagrams}\label{secp}

In this section,
we shall construct two (contravariant) functors,
$\mathcal{P}:\mathbf{SDS}\to \mathbf{OBD}_{\mathrm{po}}$
and $\mathcal{V} : \mathbf{OBD}_{\mathrm{po}}\to \mathbf{SDS}$,
which are equivalences of categories and are inverse to each other.
In particular, the functor $\mathcal{P}$ provides a model
for factor maps between Cantor minimal systems and the
functor $\mathcal{V}$ provides a method to construct factor maps
between two Cantor minimal systems by drawing suitable arrows
(i.e., ordered premorphisms) between their ordered Bratteli diagrams.

\subsection{The functor $\mathcal{P}$ from $\mathbf{SDS}$ to $\mathbf{OBD}$}
\noindent
In this subsection, we define the category of scaled essentially
minimal totally disconnected dynamical systems $\mathbf{SDS}$, and construct
a functor
$\mathcal{P}:\mathbf{SDS}\to \mathbf{OBD}_{\mathrm{po}}$ which is an
  equivalence of categories (Theorem~\ref{threqucat}).

Here we are mainly interested in  minimal (totally disconnected) systems,
but almost all the results hold in
a more general setting, namely, essentially minimal totally disconnected systems.
The minimal case will be discussed
more specifically in Section~\ref{seccms}.

\begin{definition}[cf.~\cite{hps92}, Definition~1.2]\label{defemds}
Let $X$ be a metrizable compact   space, let
$\varphi$ be a homeomorphism of $X$, and let $x_{0}\in X$. The triple $(X,\varphi,x_{0})$
is called an \emph{essentially minimal dynamical system} if the dynamical system
$(X,\varphi)$ has a unique minimal (non-empty, closed, invariant) subset $Y$ and $x_{0}\in Y$.
\end{definition}

Recall that if moreover   $X$ is totally disconnected and has no isolated points, then $X$ is homeomorphic to
the Cantor set. There are of course essentially minimal  totally disconnected
dynamical systems which are not minimal.
For example, the one-point compactification of a locally compact
non-compact Cantor minimal system
is essentially minimal but not minimal (\cite{ma02}).

\begin{definition}\label{defcatds}
Let us define the category $\mathbf{DS}$ of essentially minimal  totally disconnected dynamical systems
as follows. The objects of this category are  essentially minimal  totally disconnected dynamical systems.
Let $(X,\varphi,x_{0})$ and $(Y,\psi,y_{0})$ be in $\mathbf{DS}$. By a morphism
$\alpha :(X,\varphi,x_{0})\to (Y,\psi,y_{0})$ in $\mathbf{DS}$ we shall mean a homomorphism from
the dynamical system $(X,\varphi)$ to $(Y,\psi)$ (i.e., a continuous map
$\alpha : X \to Y$
with $\alpha\circ \varphi=\psi\circ \alpha$) such that $\alpha(x_{0})=y_{0}$.
\end{definition}

Note that in the definition above, $\alpha$ maps the unique minimal subset of $(X,\varphi)$ to
that of $(Y,\psi)$. Also, isomorphism in the category $\mathbf{DS}$ coincides with
pointed topological conjugacy introduced in \cite{hps92}.
We recall the notion of a Kakutani--Rokhlin partition \cite{hps92}.

\begin{definition}\label{defkrp}
Let $(X,\varphi, x_{0})$ be an essentially minimal  totally disconnected dynamical system.
A \emph{Kakutani--Rokhlin partition} for $(X,\varphi, x_{0})$ is a partition $P$ of $X$ where
\begin{equation*}
P=\{Z(k,j) \mid k=1,\ldots,K,\  j=1,\ldots, J(k)\},
\end{equation*}
in which $K$ and $J(1),\ldots, J(K)$ are non-zero positive  integers and the $Z(k,j)$ are non-empty clopen
subsets of $X$ with the following properties:
\begin{enumerate}
\item\label{defkrp_it1}
$\varphi(Z(k,j))=Z(k,j+1)$ for all  $1\leq k\leq K$, and $1\leq j<J(k)$;
\item\label{defkrp_it2}
setting $Z=\bigcup_{k}Z(k, J(k))$, one has $x_{0} \in Z$ and
$\varphi (Z)=\bigcup_{k}Z(k,1)$.
\end{enumerate}
For each $1\leq k\leq K$, the set $\{Z(k,j) \mid   j=1,\ldots, J(k)\}$ is called the $k$th \emph{tower}
of $P$ with \textit{height} $J(k)$. The sets  $Z$ and $\varphi(Z)$  are called
 the \emph{top} and  \emph{base} of $P$, respectively.
\end{definition}

The following definition was used implicitly in \cite{hps92}.

\begin{definition}\label{defskrp}
Let $(X,\varphi, x_{0})$ be an essentially minimal totally disconnected dynamical system.
A \emph{system of Kakutani--Rokhlin partitions} for $(X,\varphi, y)$ is a sequence
$({P}_{n})_{n=0}^{\infty}$ of Kakutani--Rokhlin partitions for $X$   such that $P_{0}=\{X\}$ and:
\begin{enumerate}
\item\label{defskrp_it1}
if $Z_{n}$ denotes the top of $P_n$ for each $n\geq 1$,
the sequence $(Z_{n})_{n=1}^{\infty}$ is a decreasing\\
\hspace*{-1.4cm}
sequence of clopen sets with intersection $\{x_{0}\}$;
\item\label{defskrp_it2}
for all $n$, $P_{n}\leq P_{n+1}$, i.e., $P_{n+1}$ is a refinement of $P_n$;
\item\label{defskrp_it3}
$\bigcup_{n=0}^{\infty}P_{n}$ is a basis for the topology of $X$.
\end{enumerate}
\end{definition}

\begin{definition}\label{defSDS}
By a \emph{scaled essentially minimal totally disconnected dynamical system} we mean a quadruple
$(X,\varphi,x_{0},\mathcal{R})$ where $(X,\varphi, x_{0})$ is an essentially minimal totally disconnected
dynamical system and $\mathcal{R}$ is a system of Kakutani--Rokhlin partitions for
$(X,\varphi, x_{0})$. The \textit{category of scaled essentially minimal totally disconnected
dynamical systems} $\mathbf{SDS}$
is the category whose objects are the  essentially minimal totally disconnected dynamical systems and
whose morphisms are as follows.
Let $(X,\varphi,x_{0},\mathcal{R})$ and $(Y,\psi,y_{0},\mathcal{S})$ be in $\mathbf{SDS}$.
By a morphism
$\alpha :(X,\varphi,x_{0},\mathcal{R})\to (Y,\psi,y_{0},\mathcal{S})$ we mean a homomorphism
between the dynamical systems $(X,\varphi)$ and $(Y,\psi)$ (i.e., a continuous map
$\alpha:X\to Y$
with $\alpha\circ \varphi=\psi\circ \alpha$) such that $\alpha(x_{0})=y_{0}$.
\end{definition}

We shall need the following notation in a number of places.

\begin{notation}
Let $(X,\varphi,x_{0})$ be an essentially minimal totally disconnected dynamical system
and $P$ and $Q$ be a pair of Kakutani--Rokhlin partitions for it such that
$P\leq Q$ (i.e., $Q$ is a refinement of $P$) and
the top of $P$ contains the top of $Q$.
Considering the towers of $P$ and $Q$ as vertices,
we shall denote by $E(P,Q)$ the (ordered) edge set from
$P$ to $Q$  defined as follows.
We have an edge in $E(P,Q)$ each time a tower of $Q$ passes
a tower of $P$; explicitly, $E(P,Q)$ contains all elements of the form
$(S,T,k)$ where $S=\{Z_{1},\ldots,Z_{n}\}$ and
$T=\{Y_{1},\ldots,Y_{m}\}$ are towers of $P$ and $Q$, respectively,
and $k$ is a positive (i.e., non-negative) integer such that
$Y_{k+j}\subseteq Z_{j}$ for all $1\leq j\leq n$ (cf.~\cite[Section~4]{hps92}).
Note that $(S,T,k)\in E(P,Q)$ if and only if
$Y_{k+j}\subseteq Z_{j}$ for some $1\leq j\leq n$.
Write $(S,T,k)\leq (S',T',k')$ if $T=T'$ and $k\leq k'$.
This is an order relation on  $E(P,Q)$, which is a total order
on the subset of edges leading to a common vertex.
\end{notation}

We shall  need the following lemma.
This is a topological version of \cite[Lemma~3.4]{aeg14}
(see Definition~\ref{defpreobd} for the notation
$\cong$). The proof is straightforward.

\begin{lemma}\label{lemtop}
Let $(X,\varphi,x_{0})$ be an essentially minimal totally disconnected dynamical system
and let $P_{1}$, $P_{2}$, and $P_{3}$ be Kakutani--Rokhlin partitions
such that $P_{1}\leq P_{2}\leq P_{3}$ and the top of $P_{i}$ contains
the top of $P_{i+1}$, for $i=1,2$.
Then  $E(P_{1},P_{3})\cong E(P_{1},P_{2})\circ E(P_{2},P_{3})$, i.e.,
the following diagram commutes, in the natural sense:

\[
\xymatrix{
P_{1}\ar[rr]^{E(P_{1},P_{2})} \ar[rrd]_{E(P_{1},P_{3})}
&  &P_{2}\ar[d]^{E(P_{2},P_{3})}\\
& & P_{3}\ .
}
\]
\end{lemma}

Now we are ready to define the functor $\mathcal{P}: \mathbf{SDS}\to\mathbf{OBD}$.
(The definition of this functor on  objects was already given in \cite[Section~4]{hps92}.)

Define the contravariant functor $\mathcal{P}: \mathbf{SDS}\to\mathbf{OBD}$
 as follows.
Let $(X,\varphi,x_{0},\mathcal{R})$ be in $\mathbf{SDS}$. Consider
  the ordered
 Bratteli diagram $\mathcal{P}(X,\varphi,x_{0},\mathcal{R})=(V,E,\geq)$
 constructed in \cite[Section~4]{hps92} for $(X,\varphi,x_{0},\mathcal{R})$.
Let $\mathcal{R}$ be as in Definition~\ref{defskrp} and set
$V_{n}=\{(n,T) \mid T \text{ is a tower of } P_{n}\}$,  $n\geq 0$, and
 $V=\bigcup_{n=0}^{\infty}V_{n}$. Set
$E_{n}=\{(n,S,T,k)\mid (S,T,k)\in E(P_{n-1},P_{n})\}$,  $n\geq 1$, and
$E=\bigcup_{n=1}^{\infty}E_{n}$. The  order on $E$ is defined as the union of
orderings on the $E_{n}$ as described
just before Lemma~\ref{lemtop}.

Now let
$(X,\varphi,x_{0},\mathcal{R})$ and  $(Y,\psi,y_{0},\mathcal{S})$ be in
$\mathbf{SDS}$, where $\mathcal{R}=(P_{n})_{n=0}^{\infty}$
and $\mathcal{S}=(Q_{n})_{n=0}^{\infty}$, and let
$\alpha :(X,\varphi,x_{0},\mathcal{R})\to (Y,\psi,y_{0},\mathcal{S})$ be a morphism
in $\mathbf{SDS}$, i.e., a continuous map $\alpha: X \to Y$ with $\alpha(x_{0})=y_{0}$ and
$\alpha\circ \varphi=\psi\circ\alpha$
(no relation to $\mathcal{R}$ and $\mathcal{S}$).
Define the (ordered) premorphism
$f=(F, (f_{n})_{n=0}^{\infty},\geq)$
from
$\mathcal{P}(Y,\psi,y_{0},\mathcal{S})=(W,S,\geq)$ to
$\mathcal{P}(X,\varphi,x_{0},\mathcal{R})=(V,E,\geq)$
as follows. Set $f_{0}=0$ and $F_{0}=\{0\}$, and suppose that
we have chosen $f_{0},f_{1},\ldots,f_{n-1}$ and $F_{0},F_{1},\ldots,F_{n-1}$.
To define $f_{n}$ and $F_{n}$, observe that
since $Q_{n}$ is a Kakutani--Rokhlin partition for $(Y,\psi,y_{0})$, the set $\alpha^{-1}(Q_{n})$
of inverse images of elements of $Q_{n}$ is a
Kakutani--Rokhlin partition for $(X,\varphi,x_{0})$.
By  Properties~\eqref{defskrp_it1} and \eqref{defskrp_it3} of Definition~\ref{defskrp},
there is an integer $f_{n}$ with $f_{n}\geq f_{n-1}$  such that
$\alpha^{-1}(Q_{n})\leq P_{f_{n}}$, the top of
$\alpha^{-1}(Q_{n})$ contains the top of $P_{f_{n}}$, and
the  sequence $(f_{n})_{n=0}^{\infty}$ is cofinal.
Set
\[
F_{n}=\{(n,S,T,k)\mid (\alpha^{-1}(S),T,k)\in E(\alpha^{-1}(Q_{n}), P_{f_{n}})\}.
\]
There is a natural one-to-one correspondence between
$F_{n}$ and
$E(\alpha^{-1}(Q_{n}), P_{f_{n}})$.
Define the order on $F_{n}$ to be the induced order from
$E(\alpha^{-1}(Q_{n}), P_{f_{n}})$.
This makes $F_{n}$ an edge set from $W_{n}$ to $V_{f_{n}}$.

Continuing this procedure, we can obtain a cofinal sequence of
integers $(f_{n})_{n=0}^{\infty}$ with $f_{0}=0\leq f_{1}\leq f_{2}\leq\cdots$
 and
an  edge set
$F=\bigcup_{n=0}^{\infty} F_{n}$ such that each $F_{n}$ is an
edge set from $W_{n}$ to $V_{f_{n}}$. The source and range   maps
are defined in the natural way, i.e., $s(n,S,T,k)=(n,S)$
and $r(n,S,T,k)=(f_{n},T)$.
The  order $\leq$ on $F$ is the union of the  orders on  $F_{n}$.
Now set $f=(F, (f_{n})_{n=0}^{\infty},\geq)$. Applying Lemma~\ref{lemtop},
 we see that $f:(W,S)\to (V,E)$ is an (ordered) premorphism.
Set $\mathcal{P}(\alpha)=[f]$, the equivalence class of
$f$. The following is immediate.

 \begin{proposition}
 The map $\mathcal{P}: \mathbf{SDS}\to\mathbf{OBD}$
 is a  contravariant functor.
\end{proposition}

Next, we show that
any premorphism between the Bratteli diagrams of two
essentially minimal totally disconnected dynamical systems can be lifted  to a homomorphism
between them. 

\begin{theorem}\label{thrful}
 The functor $\mathcal{P}: \mathbf{SDS}\to\mathbf{OBD}$
 is a full and faithful functor.
\end{theorem}
\begin{proof}
First let us show that $\mathcal{P}$ is full. The idea  is to reverse the procedure described above.
Let $\mathcal{X}_{1}=(X,\varphi,x_{0},\mathcal{R})$ and
$\mathcal{X}_{2}=(Y,\psi,y_{0},\mathcal{S})$
be in $\mathbf{SDS}$ and write
$\mathcal{P}(\mathcal{X}_{1})=(V,E)$ and
$\mathcal{P}(\mathcal{X}_{2})=(W,S)$.
Let $f:(W,S)\to (V,E)$ be an (ordered) premorphism.
We must show that there is a morphism
$\alpha:\mathcal{X}_{1}\to \mathcal{X}_{2}$ with
$\mathcal{P}(\alpha)=[f]$.

Write $f=(F, (f_{n})_{n=0}^{\infty},\geq)$,
$\mathcal{R}=(P_{n})_{n=0}^{\infty}$, and
$\mathcal{S}=(Q_{n})_{n=0}^{\infty}$.
Let $F=\bigcup _{n=0}^{\infty} F_{n}$ denote the decomposition of $F$ according to
Definition~\ref{defpreobd}.
For each $n\geq 0$, $F_{n}$ fills the towers of $P_{f_{n}}$ with the towers of $Q_{n}$;
specifically, let $T$ be a tower of $P_{f_{n}}$. Let $e_{1},e_{2},\ldots, e_{k}$ denote
the edges in $F_{n}$ with range $(f_{n},T)$ and
$e_{1}<e_{2}<\cdots< e_{k}$.
Denote by $S_{i}$  the tower of $Q_{n}$ such that the source of $e_{i}$ is
$(n,S_{i})$, $1\leq i\leq k$. Then the height of $T$ equals the sum of
the heights of  $S_{1},S_{2},\ldots, S_{k}$, since $f$ is a premorphism.

Choose $x$ in $X$. For each $n\geq 0$ there is $A_{n}\in P_{f_{n}}$ such that $x\in A_{n}$.
We have $A_{0}\supseteq A_{1}\supseteq A_{2}\supseteq\cdots$.
Since $X$ is Hausdorff and $\bigcup_{n=0}^{\infty} P_{n}$ is a basis for $X$,
we have $\bigcap_{n=0}^{\infty}A_{n}=\{x\}$.
Fix $n\geq 0$. For $T_{n}$,  the tower of $P_{f_{n}}$ containing $A_{n}$,
by the preceding paragraph, there is a unique tower $S_{n}$ in $Q_{n}$ and a unique element $B_{n}$
in $S_{n}$ which corresponds to $A_{n}$ when $F_{n}$ fills $T_{n}$
by the towers of $Q_{n}$. We may construct $\alpha$ in such a way that
$\alpha(A_{n})\subseteq B_{n}$.
By Definition~\ref{defpreobd}, for each $n\geq 1$,
we have
$E_{n}\circ F_{n}\cong F_{n-1}\circ S_{f_{n-1},f_{n}}$.
Thus, $B_{n}\subseteq B_{n-1}$.
The set
$\bigcap_{n=0}^{\infty}B_{n}$ is a singleton, say with the  element  $\alpha(x)$.
This gives a map $\alpha:X\to Y$.

Our construction yields  $\alpha(A_{n})\subseteq B_{n}$, for $n\geq 0$. From this, it follows that
$\alpha$ is continuous. Let us show that $\alpha(x_{0})=y_{0}$.
Since $y$ is in the top of each $P_{f_{n}}$,
$\alpha(x_{0})$ is in the top of each $Q_{n}$. Now by  Property~\eqref{defskrp_it1} of Definition~\ref{defskrp}
we have $\alpha(x_{0})=y_{0}$.

It is not hard to see that $\alpha\circ \varphi=\psi\circ 	\alpha$. Let $x\in X\setminus \{x_{0}\}$.
Hence, $\alpha:\mathcal{X}_{1}\to \mathcal{X}_{2}$ is a
morphism in $\mathbf{SDS}$ and our construction shows that
$\alpha^{-1}(Q_{n})\leq P_{f_{n}}$, $n\geq 0$.
Moreover, the premorphism associated to $\alpha$ for the sequence
$(f_{n})_{n=0}^{\infty}$
is obviously equivalent to $f$. Hence,
$\mathcal{P}(\alpha)=[f]$.
The proof of the faithfulness of $\mathcal{P}$ is straightforward.
\end{proof}

Let us determine the essential range of the functor
$\mathcal{P}: \mathbf{SDS}\to\mathbf{OBD}$.
 Recall that the essential range of a functor
  is the subcategory of those objects in the
   codomain
category which are isomorphic to  objects in the range of the functor.

Let us denote by $\mathbf{OBD}_{\mathrm{po}}$
the full subcategory of $\mathbf{OBD}$ consisting of all  properly ordered Bratteli diagrams
(see Definition~\ref{defpobd}).

\begin{lemma}\label{lemessran}
The essential range of  $\mathcal{P}: \mathbf{SDS}\to\mathbf{OBD}$
is $\mathbf{OBD}_{\mathrm{po}}$.
\end{lemma}
\begin{proof}
Let $\mathcal{X}$
be in $\mathbf{SDS}$. Then the ordered Bratteli diagram
$\mathcal{P}(\mathcal{X})$
is properly ordered \cite[Section~4]{hps92}.
Now let $B$ be an ordered Bratteli diagram which is isomorphic in $\mathbf{OBD}$ to
$\mathcal{P}(\mathcal{X})$, for some
$\mathcal{X}\in\mathbf{SDS}$.
By Proposition~\ref{thrcatobd}, $B$ is equivalent to
$\mathcal{P}(\mathcal{X})$. By \cite[Proposition~2.7]{hps92},
$B$ is also properly ordered. Hence the essential range of
$\mathcal{P}$ is contained in
$\mathbf{OBD}_{\mathrm{po}}$. Now
let $B$ be a properly ordered Bratteli diagram.
Denote by $(X,\varphi,x_{0})$  the Vershik transformation associated to
$B$ \cite[Section~3]{hps92}. Fix a system of Kakutani--Rokhlin partitions
$\mathcal{R}$ for
$(X,\varphi,x_{0})$, which exists by \cite[Theorem~4.2]{hps92}.
By \cite[Theorem~4.6]{hps92},
$B$ is equivalent to $\mathcal{P}(X,\varphi,x_{0},\mathcal{R})$.
\end{proof}

The following result  follows from
Theorem~\ref{thrful},
  Lemma~\ref{lemessran}, and
  \cite[Theorem~IV.4.1]{ma98}.

\begin{theorem}\label{threqucat}
 The functor $\mathcal{P}: \mathbf{SDS}\to\mathbf{OBD}_{\mathrm{po}}$
 is an equivalence of categories.
\end{theorem}

\subsection{The functor $\mathcal{V}$ from $\mathbf{OBD}_{\mathrm{po}}$ to $\mathbf{SDS}$}\label{subsecv}
\noindent
In this subsection we shall construct the contravariant functor
$\mathcal{V} : \mathbf{OBD}_{\mathrm{po}}\to \mathbf{SDS}$ which is
the inverse of the functor
$\mathcal{P}: \mathbf{SDS}\to\mathbf{OBD}_{\mathrm{po}}$.

The definition
of  $\mathcal{V}$ on objects of $\mathbf{OBD}_{\mathrm{po}}$
coincides with the construction in \cite{hps92}. Our contribution is finding the way $\mathcal{V}$ acts on
 morphisms. In particular, this gives a way to construct factor maps between two Cantor
minimal systems by drawing suitable arrows between their ordered Bratteli diagrams.

Let $B=(V,E,\geq)$ be a properly ordered Bratteli diagram.
Denote by $\mathcal{V}(B)$ the  Bratteli--Vershik dynamical
system
associated
to $B$, as described in \cite[Section~3]{hps92} and  \cite[Section~3]{gps95}.
Recall that $\mathcal{V}(B)$ is defined as follows.
Let $X_{B}$ denote the space of infinite paths,
 topologized by specifying a basis of open sets, namely the family of
cylinder sets $U(e_{1},e_{2},\ldots,e_{k})=\{(f_{1},f_{2},\ldots) \mid
f_{i}=e_{i}, \ 1\leq i\leq k\}$.
 Denote by $x_{\max}$ and $x_{\min}$ the unique elements of
$E_{\max}$ and $E_{\min}$, respectively. The homeomorphism
$\lambda_{B}:X_{B}\to X_{B}$, called the Vershik transformation,
is defined in \cite[Section~3]{hps92}.
Then   $(X_{B},\lambda_{B},x_{\max})$
is an essentially minimal totally disconnected dynamical system.

Let us recall from \cite{hps92} the canonical system of Kakutani--Rokhlin partitions
$\mathcal{R}_{B}=(P_{n})_{n=0}^{\infty}$ for $(X_{B},\lambda_{B},x_{\max})$
such that $(X_{B},\lambda_{B},x_{\max}, \mathcal{R}_{B})$ is in $\mathbf{SDS}$.
Set $P_{0}=\{X_{B}\}$. Fix $n\geq 1$ and define $P_{n}$ as follows.
For each $v\in V_{n}$ we have a tower
$T_{v}$ in $P_{n}$.
For each $(e_{1},e_{2},\ldots,e_{n})$ in
$E_{1}\circ E_{2}\circ\cdots\circ E_{n}$ with
$r(e_{n})=v$ we have an element
$U(e_{1},e_{2},\ldots,e_{n})$, as defined above, in $T_{v}$. Hence,
\[
P_{n}=\{U(e_{1},e_{2},\ldots,e_{n}) \mid (e_{1},e_{2},\ldots,e_{n})\in
E_{1}\circ E_{2}\circ\cdots\circ E_{n}\}.
\]
Note that each $P_{n}$ is a Kakutani--Rokhlin partition and
that $\mathcal{R}_{B}=(P_{n})_{n=0}^{\infty}$ satisfies the conditions of
Definition~\ref{defskrp}, and hence is a system of Kakutani--Rokhlin partitions for
$(X_{B},\lambda_{B},x_{\max})$. Finally, set
$\mathcal{V}(B)=(X_{B},\lambda_{B},x_{\max},\mathcal{R}_{B})$.
To summarize:

\begin{proposition}
For each ordered Bratteli diagram $B=(V,E,\geq)$, the system
$\mathcal{V}(B)=(X_{B},\lambda_{B},x_{\max},\mathcal{R}_{B})$
is in $\mathbf{SDS}$.
\end{proposition}

Let $B=(V,E,\geq)$ be an ordered Bratteli diagram. Define an ordered premorphism
$f_{B}:B\to \mathcal{P}(\mathcal{V}(B))$
as follows:
$f_{B}=(F_{B}, (n)_{n=0}^{\infty},\geq)$, where
$F_{B}=\{(v,T_{v})\mid v\in V\}$.
The decomposition of $F_{B}$  is obtained by setting
$F_{B,n}=\{(v,T_{v})\mid v\in V_{n}\}$, $n\geq 0$. The source and range maps of $F_{B}$ are defined by
$s(v,T_{v})=v$ and $r(v,T_{v})=T_{v}$. There is a unique way to define an  order on $F_{B}$ as above,
since $r^{-1}\{T_{v}\}$ is a singleton.
It is not hard to see that
$f_{B}:B\to \mathcal{P}(\mathcal{V}(B))$ is an ordered premorphism,
which turns to be  an isomorphism in the category of ordered Bratteli diagrams with ordered premorphisms
(see Definition~\ref{defisopre}).
Denote by $\tau_{B}:B\to \mathcal{P}(\mathcal{V}(B))$
the associated ordered morphism, i.e.,
$\tau_{B}=[f_{B}]$, which is an isomorphism in $\mathbf{OBD}$.

Once one fixes the isomorphism $\tau_{B}=[f_{B}]$, for each $B$, there is a unique way to
define $\mathcal{V} : \mathbf{OBD}_{\mathrm{po}}\to \mathbf{SDS}$ on morphisms
to obtain the natural inverse of $\mathcal{P}: \mathbf{SDS}\to\mathbf{OBD}_{\mathrm{po}}$
(see the proof of \cite[Theorem~IV.4.1]{ma98} for details). In fact,
let $h:B\to C$ be a morphism in
$\mathbf{OBD}_{\mathrm{po}}$.
Then $\tau_{C}h\tau_{B}^{-1}: \mathcal{P}(\mathcal{V}(B))\to
\mathcal{P}(\mathcal{V}(C))$ is a morphism in
$\mathbf{OBD}_{\mathrm{po}}$. By Theorem~\ref{thrful}, there is a
unique morphism $\alpha: \mathcal{V}(C)\to \mathcal{V}(B)$ such that
$\mathcal{P}(\alpha)=h$. Set $\mathcal{V}(h)=\alpha$.
We have almost finished showing the next result. All the required properties of the map $\mathcal{V}$ follow from the equality
$\mathcal{P}(\mathcal{V}(f))\tau_{B}=\tau_{A}h$  
(cf.~the proof of \cite[Theorem~IV.4.1]{ma98}), which, in particular, gives
 $\tau: 1_{{\mathbf{OBD}}_{\mathrm{es}}}\cong \mathcal{P}\mathcal{V}$.

\begin{theorem}\label{thrfuncv}
 The map $\mathcal{V} : \mathbf{OBD}_{\mathrm{po}}\to \mathbf{SDS}$ as defined above,
 is a contravariant functor which is
  an equivalence of categories and the unique (up to natural isomorphism) inverse of
  the functor $\mathcal{P}:\mathbf{SDS}\to \mathbf{OBD}_{\mathrm{po}}$.
\end{theorem}

Let us examine how the functor $\mathcal{V}$ acts on morphisms.
Let $f:B\to C$ be an ordered premorphism as in Definition~\ref{defpreobd}.
Define a map $\alpha:X_{C}\to X_{B}$ as follows.
Let $x=(s_{1},s_{2},\ldots)$ be in $X_{C}$, i.e., an infinite path in $S$.
Define the path $\alpha(x)=(e_{1},e_{2},\ldots)$ in $X_{B}$ as follows.
Fix $n\geq 1$.
By Definition~\ref{defpreobd}, the  diagram

 \[
\xymatrix{V_{0}\ar[r]^{E_{0,n}}\ar[d]_{F_{0}}
 &V_{n}\ar[d]^{F_{n}} \\
 W_{0}\ar[r]_{S_{0,f_{n}}}
 &W_{f_{n}}
 }
\]
commutes,  that is, $F_{0}\circ S_{0,f_{n}}\cong E_{0,n}\circ F_{n}$. Thus, there is a unique path
$(e_{1},e_{2},\ldots,e_{n},d_{n})$ in $E_{0,n}\circ F_{n}$,  corresponding to
the path $(s_{0},s_{1},\ldots,s_{f_{n}})$ in
$F_{0}\circ S_{0,f_{n}}$,
where $s_{0}$ is the unique element of $F_{0}$.
We need to check that the first $n$ edges of the path associated to each $m>n$ coincide with the edges for $n$.

\begin{lemma}\label{lemtech}
With the above notation, let $m>n$ and
 consider the path $(e_{1}',e_{2}',\ldots,e_{m}',d_{m}')$ in
$E_{0,m}\circ F_{m}$
associated
to $m$ in the above construction.
Then  $e_{i}'=e_{i}$ for each $1\leq i\leq n$.
\end{lemma}
\begin{proof}
Consider the following commutative diagram:
\[
\xymatrix{
V_{0}\ar[r]^{E_{0,n}}\ar[d]_{F_{0}}
 &V_{n}\ar[d]^{F_{n}}\ar[r]^{E_{n,m}} & V_{m}\ar[d]^{F_{m}} \\
 W_{0}\ar[r]_{S_{0,f_{n}}}
 & W_{f_{n}} \ar[r]_{S_{f_{n},f_{m}}} & W_{f_{m}} .
 }
\]
Since $(e_{1},e_{2},\ldots,e_{n},d_{n})$ in $E_{0,n}\circ F_{n}$
is the unique path corresponding to $(s_{0},s_{1},\ldots,s_{f_{n}})$ in
$F_{0}\circ S_{0,f_{n}}$, we get $r(d_{n})=r(s_{f_{n}})$. Thus the  path
\begin{equation}\label{equ1}
(e_{1},e_{2},\ldots,e_{n},d_{n},s_{f_{n}+1},\ldots,s_{f_{m}})
\end{equation}
in $ E_{0,n}\circ F_{n} \circ S_{f_{n},f_{m}}$
is the unique path corresponding to
$(s_{0},s_{1},\ldots,s_{f_{m}})$ in $F_{0}\circ S_{0,f_{m}}$ by the isomorphism
$ E_{0,n}\circ F_{n} \circ S_{f_{n},f_{m}}\cong F_{0}\circ S_{0,f_{m}}$.

Moreover, the path
$(e_{1}',e_{2}',\ldots,e_{m}',d_{m}')$  in
$E_{0,m}\circ F_{m}$ is the unique path corresponding to
the path $(s_{0},s_{1},\ldots,s_{f_{m}})$
by the isomorphism
$ E_{0,m}\circ F_{m} \cong F_{0}\circ S_{0,f_{m}}$.
Since the  isomorphisms involved are unique (because of the order),
by the isomorphism
$ E_{0,m}\circ F_{m} \cong
E_{0,n}\circ F_{n} \circ S_{f_{n},f_{m}}$,
the path $(e_{1}',e_{2}',\ldots,e_{m}',d_{m}')$
corresponds to
the path in  (\ref{equ1}).

Let $(e_{n+1}'',e_{n+2}'',\ldots,e_{m}'',d_{m}'')$ denote the unique path in
$ E_{n,m}\circ F_{m}$ corresponding to the path
$(d_{n},s_{f_{n}+1},\ldots,s_{f_{m}})$ in $F_{n}\circ S_{f_{n},f_{m}}$
by the isomorphism
 $ E_{n,m}\circ F_{m}  \cong F_{n}\circ S_{f_{n},f_{m}}$.
Then $r(e_{n})=s(d_{n})=s(e_{n+1}'')$. Thus, the path
\begin{equation}\label{equ2}
(e_{1},e_{2},\ldots,e_{n},e_{n+1}'',e_{n+2}'',\ldots,e_{m}'',d_{m}'')
\end{equation}
corresponds to the path in \eqref{equ1} by the isomorphism
$ E_{0,m}\circ F_{m} \cong E_{0,n}\circ F_{n}\circ S_{f_{n},f_{m}}$.
Since the path $(e_{1}',e_{2}',\ldots,e_{m}',d_{m}')$
also corresponds to
the path~(\ref{equ1}) by the same isomorphism, we conclude that
 $(e_{1}',e_{2}',\ldots,e_{m}',d_{m}')$ is equal to the path in
\eqref{equ2}.
Therefore, $e_{i}'=e_{i}$ for  $1\leq i\leq n$,
$e_{i}'=e_{i}''$ for  $n+1\leq i\leq m$, and $d_{m}'=d_{m}''$.
\end{proof}

By the preceding lemma, one can define, without ambiguity,
the path $\alpha(x)=(e_{1},e_{2},\ldots)$ in $X_{B}$
associated to the path $x=(s_{1},s_{2},\ldots)$  in $X_{C}$.
(We will describe a second way  to compute $\alpha(x)$  after Proposition~\ref{propcom}.)
We thus obtain  a map
$\alpha:X_{C}\to X_{B}$.
 It is not hard to see that if we replace $f$ with another representative
of the class $[f]$, then we get the same $\alpha$.

\begin{proposition}\label{propcom}
Let $f:B\to C$ be a premorphism in $\mathbf{OBD}_{\mathrm{po}}$
and $\alpha:X_{C}\to X_{B}$ be its associated map
 as defined above. Then $\mathcal{V}([f])=\alpha$.
\end{proposition}

\begin{proof}
With the above notation,  $\alpha(U(s_{1},s_{2},\ldots,s_{f_{n}}))
\subseteq U(e_{1},e_{2},\ldots,e_{n})$. This shows that $\alpha$ is continuous.
Moreover,  $\alpha\circ \lambda_{C}=\lambda_{B}\circ \alpha$
and $\alpha$ maps the unique path in $S_{\max}$ to the unique path in $E_{\max}$.
Thus, $\alpha:\mathcal{V}(C)\to \mathcal{V}(B)$ is a morphism in $\mathbf{SDS}$.
Also,
$\mathcal{P}(\alpha)\tau_{B}=\tau_{A}[f]$. Hence,
$\mathcal{P}(\alpha)=\mathcal{P}(\mathcal{V}([f]))$,
and so by Theorem~\ref{thrful},
 $\mathcal{V}([f])=\alpha$.
\end{proof}

We remark that the proof of Lemma~\ref{lemtech} gives
a second  method for computing $\mathcal{V}([f])$ above.
This turns out to be easier to follow---at least in some cases---as it requires less
computation.
In fact, let $x=(s_{1},s_{2},\ldots)$ be in $X_{C}$.
Define the path $\alpha(x)=(e_{1},e_{2},\ldots)$ in $X_{B}$ as follows.
First consider the following commutative diagram:
 \[
\xymatrix{V_{0}\ar[r]^{E_{1}}\ar[d]_{F_{0}}
 &V_{1}\ar[d]^{F_{1}} \\
 W_{0}\ar[r]_{S_{0,f_{1}}}
 &W_{f_{1}}.
 }
\]
Then there is a unique path
$(e_{1},d_{1})$ in $E_{1}\circ F_{1}$,  corresponding to
the path $(s_{0},s_{1},\ldots,s_{f_{1}})$ in
$F_{0}\circ S_{0,f_{1}}$,
where $s_{0}$ is the unique element of $F_{0}$.
Now consider the following commutative diagram:
 \[
\xymatrix{V_{1}\ar[r]^{E_{2}}\ar[d]_{F_{1}}
 &V_{2}\ar[d]^{F_{2}} \\
 W_{f_{1}}\ar[r]_{S_{f_{1},f_{2}}}
 &W_{f_{2}}.
 }
\]
Then there is a unique path
$(e_{2},d_{2})$ in $E_{2}\circ F_{2}$,  corresponding to
the path $(d_{1},s_{f_{1}+1},\ldots,s_{f_{2}})$ in
$F_{1}\circ S_{f_{1},f_{2}}$. Continuing this procedure, we obtain
a path $(e_{1},e_{2},\ldots)$ in $X_{B}$ which is the same as
$\alpha(x)$ obtained by the previous method (by the proof of Lemma~\ref{lemtech}).

As stated in the proof of Theorem~\ref{thrfuncv}, the correspondence
(natural transformation) $\tau$,  defined above,
gives $1_{\mathbf{OBD}_{\mathrm{po}}}\cong \mathcal{P}\mathcal{V}$.
Using this, a standard categorical procedure gives a correspondence $\sigma$
which implements $1_{\mathbf{SDS}}\cong \mathcal{V}\mathcal{P}$.
In fact, let $\mathcal{X}$ be in $\textbf{SDS}$. Then
$\tau_{\mathcal{P}(\mathcal{X})}: \mathcal{P}(\mathcal{X})\to
\mathcal{P}(\mathcal{V}(\mathcal{P}(\mathcal{X})))$ is an isomorphism in $\mathbf{OBD}$.
By Theorem~\ref{thrful}, there is a unique isomorphism
$\sigma_{\mathcal{X}}:\mathcal{X}\to \mathcal{V}(\mathcal{P}(\mathcal{X}))$
such that
$\mathcal{P}(\sigma_{\mathcal{X}})=\tau_{\mathcal{P}(\mathcal{X})}^{-1}$.
Moreover,
$\sigma : 1_{\mathbf{SDS}}\cong \mathcal{V}\mathcal{P}$.
Let us summarize the results of this section as follows.

\begin{theorem}\label{theequcat2}
The contravariant functors
$\mathcal{P}: \mathbf{SDS}\to\mathbf{OBD}_{\mathrm{po}}$ and
$\mathcal{V} : \mathbf{OBD}_{\mathrm{po}}\to \mathbf{SDS}$
are equivalences of categories which are inverse to each other, with
respect to the natural isomorphisms
$\tau:  \mathcal{P}\mathcal{V} \cong 1_{\mathbf{OBD}_{\mathrm{po}}} $ and
$\sigma :  \mathcal{V}\mathcal{P} \cong 1_{\mathbf{SDS}}$.
\end{theorem}

\section{Cantor Minimal Systems}\label{seccms}

In this section we shall apply the results of the previous 
section to Cantor minimal dynamical systems and  their
factor maps
(thus obtaining new results as an application of the categorical
methods).

\subsection{Cantor Systems}
Recall that a dynamical system $(X,\varphi)$ is called a Cantor minimal
system if $X$ is homeomorphic to the Cantor set and $\varphi$ is a
minimal homeomorphism of $X$. These systems are of great importance
in symbolic dynamics. Every Cantor minimal system has a Bratteli--Vershik
model (see, e.g., the definition of the functor $\mathcal{P}$ on objects in Section~\ref{secp}).

Recall the definition of a simple Bratteli diagram from Definition~\ref{defsbd}. The following well-known fact follows from the results of \cite{hps92}.

\begin{proposition}\label{prop_min}
Let $B=(V,E,\geq)$ be a properly ordered Bratteli diagram. Then the following
statements are equivalent:
\begin{enumerate}
\item\label{prop_min_it1}
the system
$(X_{B},\lambda_{B})$ is minimal;
\item\label{prop_min_it2}
 $(V,E)$ is a simple Bratteli diagram.
\end{enumerate}
\end{proposition}

The preceding proposition and
Theorem~\ref{theequcat2} imply that the full subcategory of
$\mathbf{OBD}_{\mathrm{po}}$ consisting of simple
properly ordered Bratteli diagrams
is equivalent to
the full subcategory of
$\mathbf{SDS}$ (Definition~\ref{defSDS}) consisting of  scaled minimal dynamical systems on
metrizable, compact, totally disconnected  spaces. 


Recall that for a Bratteli diagram $B=(V,E)$, the Bratteli compactum $X_{B}$ is a metrizable, compact, totally disconnected  space.
Thus, to obtain a Cantor set, we need to translate the property
of having no isolated point into the language of diagrams. This is not hard and is done in the next lemma.
By the definition of the topology on $X_{B}$ (Subsection~\ref{subsecv}),
\eqref{lembcan_it2} is equivalent to having no isolated points. Thus, \eqref{lembcan_it1} is equivalent to \eqref{lembcan_it2}.
 Also,
\eqref{lembcan_it3} is just a reformulation of \eqref{lembcan_it2}.

\begin{lemma}\label{lembcan}
Let $B=(V,E)$ be a Bratteli diagram.
The following
statements are equivalent:
\begin{enumerate}
\item\label{lembcan_it1}
$X_{B}$ is homeomorphic to the Cantor set;
\item\label{lembcan_it2}
for each infinite path $x=(e_{1},e_{2},\ldots)$ in $X_{B}$ and each $n\geq 1$
there is an infinite\\
\hspace*{-1.4cm} path $y=(f_{1},f_{2},\ldots)$ with
$x\neq y$ and $e_{k}=f_{k}$, $1\leq k\leq n$;
\item\label{lembcan_it3}
for each
$n\geq 0$ and each  $v\in V_{n}$ there is $m\geq n$ and $w\in V_{m}$
such that there is\\
\hspace*{-1.4cm} a path from $v$ to $w$ and $|s^{-1}(\{w\})|\geq 2$.
\end{enumerate}
\end{lemma}

The next result follows immediately from the lemma above.

\begin{proposition}\label{probcan}
Let $B=(V,E)$ be a simple Bratteli diagram. Then the following
statements are equivalent:
\begin{enumerate}
\item\label{probcan_it1}
$X_{B}$ is homeomorphic to the Cantor set;
\item\label{probcan_it2}
$X_{B}$ is infinite;
\item\label{probcan_it3}
the set $\{n\in\mathbb{N} \mid |E_{n}|\geq 2\}$ is infinite.
\end{enumerate}
\end{proposition}

In this context,
Theorem~\ref{theequcat2} restricts as follows.

\begin{corollary}\label{cormdsc}
The full subcategory of
$\mathbf{OBD}_{\mathrm{po}}$ consisting of simple properly 
ordered Bratteli diagrams $B$ with infinite $X_B$,
is equivalent to the full subcategory of
$\mathbf{SDS}$  consisting of  scaled minimal dynamical systems on
Cantor sets. 
\end{corollary}

\subsection{Factor Maps}\label{subseccms}
\noindent
In this subsection we use the idea of  ordered  premorphism
to construct factor maps between
Cantor minimal systems. For simplicity, we only consider
factor maps on odometers. However,
an example of an extension of the Chacon system
is considered briefly at the end of this subsection
(Example~\ref{exa_Chacon2}). An objective of this subsection
is to illustrate some of our ideas concerning diagrams.
In particular, we reprove some facts on
extensions of odometers by using premorphisms (though there are
also some new results---such as  Proposition~\ref{prop_odo_uni}).
More examples and
 results in this direction can be found in \cite{GH18}.

Consider Example~\ref{exodtoe}.
Applying the
functor $\mathcal{V}$ to the class of the ordered premorphism $f$ in that example,
 we get a factor map $\mathcal{V}([f]): \mathcal{V}(C) \to \mathcal{V}(B)$
as defined before Lemma~\ref{lemtech}. In fact, $\mathcal{V}(B)$ is the maximal rational
equicontinuous factor of $\mathcal{V}(C)$ (see Theorem~\ref{thm_odo}
and \cite{GJ00}).
The idea of the this example can be used to reprove the
fact that every Cantor minimal system has a maximal rational
equicontinuous factor (possibly
trivial). Before showing this, we prove that factor maps are in one-to-one correspondence
to ordered morphisms. First, we recall the following notion.

\begin{definition}\label{defbvm}
Let $(X,\varphi)$ be a Cantor minimal system. By a \emph{Bratteli--Vershik model} for $(X,\varphi)$
we mean a properly ordered Bratteli diagram $B$ such that the associated system $(X_{B},\lambda_{B})$
is conjugate to $(X,\varphi)$. Let $x_{0}\in X$. By a
 \emph{Bratteli--Vershik model} for $(X,\varphi,x_{0})$ we mean
  a properly ordered Bratteli diagram $B$ such that $(X,\varphi,x_{0})$ is pointed topological conjugate to
 $(X_{B},\lambda_{B},x_{\max})$, i.e., there is a homeomorphism
 $\alpha:X\to X_{B}$ such that $\alpha\circ\varphi=\lambda_{B}\circ\alpha$
 and $\alpha(x_{0})=x_{\max}$
 (an isomorphism in $\mathbf{DS}$).
 Note that in the latter case,  $B$ is unique (up to equivalence).
\end{definition}

\begin{proposition}\label{propfb}
Let $(X,\varphi)$ and $(Y,\psi)$ be  Cantor minimal systems, and let $x\in X$ and $y\in Y$.
Let $C$ and $B$ be Bratteli--Vershik models for $(X,\varphi,x)$ and $(Y,\psi,y)$.
The following statements are equivalent:
\begin{enumerate}

\item\label{propfb_it1}
there is a factor map $\alpha: (X,\varphi) \to (Y,\psi)$ with $\alpha(x)=y$;

\item\label{propfb_it2}
there is an (ordered) premorphism $f$ from $B$ to $C$ (see Definition~\ref{defpreobd}).
\end{enumerate}
More precisely, there is a natural one-to-one correspondence between the set of factor maps $\alpha$ as in
\eqref{propfb_it1} and the set of equivalence classes of ordered premorphisms $f$ from $B$ to $C$
given by $\alpha=\mathcal{V}([f])$.
\end{proposition}

\begin{proof}
This  follows  from the fact that
the functor $\mathcal{V}$ is full and faithful (by Theorem~\ref{thrfuncv}).
More precisely, consider the mapping
$[f]\mapsto \mathcal{V}([f])$,
from the set of ordered morphism
from $B$ to $C$, into the set
of factor maps $\alpha: (X,\varphi) \to (Y,\psi)$ with $\alpha(x)=y$.
The fullness and faithfulness of
$\mathcal{V}$ imply respectively that this mapping
is surjective and injective.
This completes the proof.
\end{proof}

\begin{theorem}\label{thmfb}
Let $(X,\varphi)$ and $(Y,\psi)$ be  Cantor minimal systems.
The following statements are equivalent:
\begin{enumerate}

\item\label{thmfb_it1}
there is a factor map from $(X,\varphi)$ to $(Y,\psi)$;

\item\label{thmfb_it2}
there are Bratteli--Vershik models $C$ and $B$ for $(X,\varphi)$ and $(Y,\psi)$, respectively,
such that
there is an ordered premorphism from $B$ to $C$.
\end{enumerate}
\end{theorem}

\begin{proof}
\eqref{thmfb_it1}$\Rightarrow$\eqref{thmfb_it2}:
Suppose that
$\alpha: (X,\varphi) \to (Y,\psi)$
is a factor map. Choose $x\in X$ and
set  $y=\alpha(x)$.
Let  $C$ and $B$ be Bratteli--Vershik models for $(X,\varphi,x)$ and $(Y,\psi,y)$,
respectively (see
Definition~\ref{defbvm}).
Applying Proposition~\ref{propfb},
we get an  ordered  premorphism from $B$ to $C$.

\eqref{thmfb_it2}$\Rightarrow$\eqref{thmfb_it1}:
Let $C$ and $B$ be
 Bratteli--Vershik models  for $(X,\varphi)$ and $(Y,\psi)$, respectively, and
let
  $f\colon B\to C$ be an ordered premorphism.
  Applying the functor $\mathcal{V}$,
  we get a factor map
  $\mathcal{V}([f])\colon
   \mathcal{V}(C) \to \mathcal{V}(B)$
   (see Subsection~\ref{subsecv}).
   Since $(X,\varphi)$ and $(Y,\psi)$
   are respectively conjugate to
   $\mathcal{V}(C)$ and $\mathcal{V}(B)$,
   we obtain a factor
   map from $(X,\varphi)$ to $(Y,\psi)$.
\end{proof}

Let us recall the definition of an odometer, including
the trivial odometers (cf.~\cite{dow05}).

\begin{definition}\label{defodo}
Let $(k_{n})_{n=1}^{\infty}$ be a sequence in $\mathbb{N}$.
By an \emph{odometer of type} $(k_{n})_{n=1}^{\infty}$ we mean
a minimal system $(X,\varphi)$ where
\[
X=\prod_{n=1}^{\infty}\{0,1,\ldots,k_{n}-1\}
\]
and   the homeomorphism $\varphi:X\to X$ is addition of $(1, 0, 0,\ldots)$, with carrying.
\end{definition}

It is known that if $X$ is infinite (i.e., $k_{n}\geq 2$ for infinitely many $n$),
then $(X,\varphi)$ is a Cantor minimal system.
When $X$ is finite, $(X,\varphi)$ is minimal.

The following well-known result follows from Proposition~\ref{propfb}.

\begin{lemma}[\cite{GH18}]\label{lem_facodo}
Let $(X,\varphi)$ and $(Y,\psi)$ be  odometers of types $(k_{n})_{n=1}^{\infty}$ and
$(l_{n})_{n=1}^{\infty}$, respectively.
The following statements are equivalent:
\begin{enumerate}

\item\label{lem_facodo_it1}
$(X,\varphi)$  is a factor of  $(Y,\psi)$;

\item\label{lem_facodo_it2}
for each $n\geq 1$ there is an $m\geq 1$ such that
$k_{1}\cdots k_{n} \mid l_{1}\cdots l_{m}$.
\end{enumerate}
\end{lemma}

The following proposition is part of the literature.

\begin{proposition}\label{prop_conjodo}
Let $(X,\varphi)$ and $(Y,\psi)$ be  odometers of types $(k_{n})_{n=1}^{\infty}$ and
$(l_{n})_{n=1}^{\infty}$, respectively.
The following statements are equivalent:
\begin{enumerate}

\item\label{prop_conjodo_it1}
$(X,\varphi)$  and  $(Y,\psi)$ are conjugate;
\item\label{prop_conjodo_it2}
$(X,\varphi)$  and  $(Y,\psi)$ are orbit equivalent;
\item\label{prop_conjodo_it3}
$(X,\varphi)$  is a factor of  $(Y,\psi)$, and also $(Y,\psi)$ is a factor of $(X,\varphi)$.
\end{enumerate}
\end{proposition}

\begin{proof}
For the equivalence of \eqref{prop_conjodo_it1} and \eqref{prop_conjodo_it2} see, e.g., \cite{pu14}.
The equivalence of \eqref{prop_conjodo_it1} and \eqref{prop_conjodo_it3} follows from Lemma~\ref{lem_facodo}.
\end{proof}

In the following definition we associate an odometer $\mathcal{O}(B)$ to an
ordered Bratteli diagram $B$
and construct an (ordered) premorphism $f_{B}: \mathcal{O}(B)\to B$.

\begin{definition}\label{def_odo_fac}
Let $B=(V,E,\geq)$ be an ordered Bratteli diagram.
We associate to $B$ an odometer $\mathcal{O}(B)=(W,R,\geq)$ of type
$(r_{n})_{n=1}^{\infty}$ and an ordered premorphism $f_{B}: \mathcal{O}(B)\to B$ as follows.
Let $h_{n}$ be the greatest common divisor of the heights of the towers at  level $n$, $n\geq 0$,
and set $r_{n}={h_{n}}/{h_{n-1}}$, $n\geq 1$. More precisely,
write $V=\bigcup_{n=0}^{\infty}V_{n}$ and $E=\bigcup_{n=1}^{\infty}E_{n}$ as in
Definition~\ref{defbd2}. Let  $\mathrm{M}(E_{n})$ denote the multiplicity matrix
of $E_{n}$. Then $E_{0,n}=E_{1}\circ E_{2}\circ\cdots\circ E_{n}$ (the edge set from
$V_{0}$ to $V_{n}$)
is the set of towers at  level $n$, and the  column matrix
\[
\mathrm{M}(E_{0,n})= \mathrm{M}(E_{n})\cdots \mathrm{M}(E_{n-1})\mathrm{M}(E_{1})
=\left(
 \begin{smallmatrix}
 h_{n,1}\\
 h_{n,2}\\
 \vdots\\
 h_{n,k_{n}}
 \end{smallmatrix}
 \right),
\]
where the $h_{n,i}$ are non-zero positive integers and $k_{n}=|V_{n}|$,
consists of the heights of these towers.
Thus, $h_{n}=\mathrm{gcd}(h_{n,1},h_{n,2},\ldots,h_{n,k_{n}})$.
Note that $1=h_{0}\mid h_{1}\mid h_{2}\cdots$
and so the definition of $r_{n}={h_{n}}/{h_{n-1}}$ makes sense.
Let $\mathcal{O}(B)=(W,R,\geq)$  be the odometer of type
$(r_{n})_{n=1}^{\infty}$. Thus $R= \bigcup_{n=1}^{\infty}R_{n}$
and $|R_{n}|=r_{n}$, $n\geq 1$.
Now define $f_{B}: \mathcal{O}(B)\to B$ as follows.
 Set $f_{B}=(F, (n)_{n=0}^{\infty},\geq )$ (see Definition~\ref{defpreobd}), where
 $F= \bigcup_{n=0}^{\infty}F_{n}$ is defined as follows.
 Let $W= \bigcup_{n=0}^{\infty}W_{n}$ denote the set of vertices of $\mathcal{O}(B)$
 and write $W_{n}=\{w_{n} \}$, $n\geq 0$. Also, write
 $V_{n}=\{v^{n}_{1},v^{n}_{2},\ldots,v^{n}_{k_{n}}\}$.
  Set $F_{0}=\{(w_{0},v^{0}_{1})\}$. Thus, $F_{0}$ has only
 one edge going from $w_{0}$ to $v^{0}_{1}$. For $n\geq 1$ set
 \[
 F_{n}=\left\{
 (w_{n},v^{n}_{i},j) \mid 1\leq i \leq k_{n},\ 1\leq j \leq \tfrac{h_{n,i}}{h_{n}}
 \right\}.
 \]
Thus, $F_{n}$ has $\tfrac{h_{n,i}}{h_{n}}$ edges from $w_{n}$ to $v^{n}_{i}$.
Put an arbitrary linear order $\geq$ on these edges. Note that
the  order on $F$ is not important here as any two  orders on
$F$ give equivalent (ordered) premorphisms, since $\mathcal{O}(B)$
has only one vertex at each level.
 Put $f_{B}=(F, (n)_{n=0}^{\infty},\geq )$.
\end{definition}

Observe that $f_{B}=(F, (n)_{n=0}^{\infty},\geq )$ is an ordered premorphism. In fact, we have
  $\mathrm{M}  (F_{n})=\tfrac{1}{h_{n}}\mathrm{M}  (E_{0,n})$
  and $\gcd(\mathrm{M}  (F_{n}))=1$, $n\geq 0$.
  The commutativity condition in Definition~\ref{defpreobd}
amounts to  commutativity of the following diagram, $n=1,2,\ldots$:

\[
\xymatrix{W_{n-1}\ar[rr]^{F_{n-1}}\ar[d]_{R_{n}}
 & &V_{n-1}\ar[d]^{E_{n}}  \\
 W_{n}\ar[rr]_{F_{n}} &
 & V_{n}.
 }
\]
To see this, first note that ordered commutativity and unordered commutativity
of  this diagram coincide as $W_{n-1}$ has only one vertex. We have
\begin{align*}
\mathrm{M}(E_{n})\mathrm{M}(F_{n-1}) &=
\frac{1}{h_{n-1}}\mathrm{M}(E_{n})\mathrm{M}(E_{0,n-1})
=\frac{1}{h_{n-1}}\mathrm{M}(E_{0,n})\\
&=\frac{h_{n}}{h_{n-1}}\mathrm{M}(F_{n})=r_{n}\mathrm{M}(F_{n})\\
&=\mathrm{M}(F_{n}) \mathrm{M}(R_{n}).
\end{align*}

Now, we give an alternative proof for the existence and uniqueness
of the maximal (rational) equicontinuous factor of a Cantor minimal
system (cf.~\cite[Chapter~9]{Aus88}).

\begin{theorem}\label{thm_odo}
For any Cantor minimal system $(X,\varphi)$ there is a unique
(up to conjugacy) odometer $(Y,\psi)$ with the following properties:
\begin{enumerate}
\item\label{thm_odo_it1}
$(Y,\psi)$ is a factor of $(X,\varphi)$;
\item\label{thm_odo_it2}
every odometer which is a factor of $(X,\varphi)$ is also a factor of
$(Y,\psi)$.
\end{enumerate}
Moreover, there is a factor map $\alpha: (X,\varphi) \to (Y,\psi)$
such that, if $\beta: (X,\varphi) \to (Z,\eta)$ is a factor map
onto an odometer $(Z,\eta)$, then
there is a (necessarily unique) factor map $\gamma: (Y,\psi) \to (Z,\eta)$
such that $\beta=\gamma\circ \alpha$, i.e., the following diagram commutes:
\[
\xymatrix{
(X,\varphi) \ar[r]^{\alpha} \ar[rd]_{\beta}
& (Y,\psi) \ar[d]^{\gamma}\\
& (Z,\eta).
}
\]
\end{theorem}

\begin{proof}
First note that uniqueness of $(Y,\psi)$ follows from
Proposition~\ref{prop_conjodo} and \eqref{thm_odo_it2}.
Now let $(X,\varphi)$ be a Cantor minimal system.
We may assume that $(X,\varphi)=\mathcal{V}(B)$ for some
properly ordered Bratteli diagram $B=(V,E,\geq)$.
Let $\mathcal{O}(B)=(W,R,\geq)$ be the ordered Bratteli diagram of
the odometer of type $(r_{n})_{n=1}^{\infty}$ associated to
$B$ as in Definition~\ref{def_odo_fac}. Put
$(Y,\psi)=\mathcal{V}(\mathcal{O}(B))$.
Also, consider the ordered premorphism $f_{B}: \mathcal{O}(B)\to B$
as in Definition~\ref{def_odo_fac} and set
$\alpha=\mathcal{V}(f_{B})$. Thus,
$\alpha: (X,\varphi) \to (Y,\psi)$ is a factor map.
 
Suppose that $\beta: (X,\varphi) \to (Z,\eta)$ is a factor map
onto an odometer $(Z,\eta)$ of type $(s_{n})_{n=1}^{\infty}$.
We may assume that $(Z,\eta)=\mathcal{V}(C)$ for some
properly ordered Bratteli diagram
$C=(U,S,\geq)$, where
$U=\bigcup_{n=0}^{\infty}U_{n}$, $S=\bigcup_{n=1}^{\infty}S_{n}$,
$|U_{n}|=1$ for all $n\geq 0$, and $|S_{n}|=s_{n}$ for all $n\geq 1$.
Since $(Z,\eta)$ is dynamically homogeneous (i.e., for any
$z_{1},z_{2}\in Z$ there is a conjugacy from $(Z,\eta)$ to itself mapping $z_{1}$ to
$z_{2}$; see \cite{GH18}), there is a conjugacy
$\delta: (Z,\eta) \to (Z,\eta)$ such that $\delta(\beta(x_{\min}))=z_{\min}$
where $x_{\min}\in X$ and $z_{\min}\in Z$ are the unique minimal paths.
Consider the map $\tilde{\beta}=\delta\circ\beta:(X,\varphi) \to (Z,\eta)$,
 a factor map with
$\tilde{\beta}(x_{\min})=z_{\min}$.
Since the contravariant functor $\mathcal{V}$ is full (by Theorem~\ref{thrfuncv}---see also
Proposition~\ref{propfb}), there is an ordered premorphism
$g:C\to B$ such that $\mathcal{V}(g)=\tilde{\beta}$.
Write $g=(G, (m_{n})_{n=0}^{\infty},\geq )$
where $G=\bigcup_{n=0}^{\infty}G_{n}$ (see Definition~\ref{defpreobd}).
Consider the diagram of $g$:
 \[
\xymatrix{U_{0}\ar[r]^{S_{1}}\ar[d]_{G_{0}}
 &U_{1}\ar[r]^-{S_{2}}\ar[d]_{G_{1}} &U_{2}\ar[r]\ar[d]_{G_{2}} &\cdots \ \ \  \\
 V_{m_{0}}\ar[r]_{E_{m_{0},m_{1}}}
 &V_{m_{1}}\ar[r]_{E_{m_{1},m_{2}}} &V_{m_{2}}\ar[r]&\cdots \   .
 }
\]
Let us construct an ordered premorphism $h:C\to\mathcal{O}(B)$ such that
$f_{B} h \sim g$. Define $h=(H, (m_{n})_{n=0}^{\infty},\geq )$
where $H=\bigcup_{n=0}^{\infty}H_{n}$ is follows.
Note that, since $H_{n}$ needs to be an edge set from $U_{n}$ to $W_{m_{n}}$
and both these sets have only one vertex, we need only to determine
$t_{n}:=|H_{n}|$, and the order on $H_{n}$ is not important.
Put $t_{0}=0$. Fix $n\geq 1$. Since
$G_{0}\circ E_{0,m_{n}}\cong S_{1}\circ\cdots\circ S_{n}\circ G_{n}$,
we get $\mathrm{M}(E_{0,m_{n}})=\mathrm{M}(G_{n}) \mathrm{M}(S_{1}\circ\cdots\circ S_{n})
=s_{1}\cdots s_{n} \mathrm{M}(G_{n})$.
Taking the $\gcd$ of both sides we get
$h_{m_{n}}=s_{1}\cdots s_{n} \gcd(\mathrm{M}(G_{n}))$, where $h_{m_{n}}$ is as in
Definition~\ref{def_odo_fac}.
(Note that $\mathrm{M}(G_{n})$ is a column matrix.) Put
$t_{n}={h_{m_{n}}}/{s_{1}\cdots s_{n}}$. Observe that $h=(H, (m_{n})_{n=0}^{\infty},\geq )$
thus defined is an ordered premorphism, i.e., the following diagram of $h$ commutes:

 \[
\xymatrix{U_{0}\ar[r]^{S_{1}}\ar[d]_{H_{0}}
 &U_{1}\ar[r]^-{S_{2}}\ar[d]_{H_{1}} &U_{2}\ar[r]\ar[d]_{H_{2}} &\cdots \ \ \  \\
 W_{m_{0}}\ar[r]_{R_{m_{0},m_{1}}}
 &W_{m_{1}}\ar[r]_{R_{m_{1},m_{2}}} &W_{m_{2}}\ar[r]&\cdots \   .
 }
\]
In fact, for any $n\geq 1$ we have
\[
\mathrm{M}(R_{m_{n-1},m_{n}})\mathrm{M}(H_{n-1})=\frac{h_{m_{n}}}{h_{m_{n-1}}}\cdot t_{n-1}
=\frac{h_{m_{n}}}{s_{1}\cdots s_{n}}\cdot s_{n}=\mathrm{M}(H_{n})\mathrm{M}(S_{n}).
\]

Now the premorphisms
$f_{B}h=\left((H_{n}\circ F_{m_{n}})_{n=0}^{\infty}, (m_{n})_{n=0}^{\infty},\geq\right)$ and
 $g=(G, (m_{n})_{n=0}^{\infty},\geq )$ are isomorphic. In fact, for any $n\geq 1$
we have
\[
\mathrm{M}(H_{n}\circ F_{m_{n}})=\mathrm{M}(F_{m_{n}})\mathrm{M}(H_{n})
=\frac{h_{m_{n}}}{s_{1}\cdots s_{n}}\cdot \frac{1}{h_{m_{n}}} \,\mathrm{M}(E_{0,m_{n}})=
\mathrm{M}(G_{n}).
\]
Thus $f_{B}h \cong g$ and so $f_{B}h \sim g$
(see Definition~\ref{defisopre}).
Therefore,
$\mathcal{V}(f_{B}h)=\mathcal{V}(g)$ and so
$\mathcal{V}(h)\circ \alpha=\tilde{\beta}=\delta\circ\beta$. Put
$\gamma=\delta^{-1}\circ \mathcal{V}(h)$.
Then $\gamma: (Y,\psi) \to (Z,\eta)$ is a factor map
with $\beta=\gamma\circ \alpha$.
\end{proof}

The following (new) result gives a criterion for the existence of a factor map
from a Cantor minimal system to an odometer.

\begin{corollary}\label{cor_exi_odo}
Let $(X,\varphi)$ be  a Cantor minimal system  and let $(Y,\psi)$ be the unique odometer of type
$(r_{n})_{n=1}^{\infty}$ associated with $(X,\varphi)$ as in
Theorem~\ref{thm_odo}. For an odometer $(Z,\eta)$
of type $(s_{n})_{n=1}^{\infty}$, the following statements are equivalent:
\begin{enumerate}
\item\label{cor_odo_it1}
$(Z,\eta)$ is a factor of $(X,\varphi)$;
\item\label{cor_odo_it2}
$(Z,\eta)$ is a factor of $(Y,\psi)$;
\item\label{cor_odo_it3}
for each $n\geq 1$ there is an $m\geq 1$ such that
$s_{1}\cdots s_{n} \mid r_{1}\cdots r_{m}$.
\end{enumerate}
\end{corollary}

\begin{proof}
Note that the sequence $(r_{n})_{n=1}^{\infty}$ is defined as in
Definition~\ref{def_odo_fac} in which $B$ is
a Bratteli--Vershik model (unique up to conjugacy)
for $(X,\varphi)$. The equivalence of \eqref{cor_odo_it1} and
\eqref{cor_odo_it2} follows from Theorem~\ref{thm_odo}.
Also, the equivalence of \eqref{cor_odo_it2} and
\eqref{cor_odo_it3} follows from Lemma~\ref{lem_facodo}.
\end{proof}

Finally, we obtain the following new uniqueness result.

\begin{proposition}\label{prop_odo_uni}
Let $(X,\varphi)$ be  a Cantor minimal system  and let $(Z,\eta)$ be an  odometer.
Then there is at most one factor map (up to conjugacy)  from $(X,\varphi)$
to $(Z,\eta)$; that is, if
$\beta_{1}, \beta_{2} : (X,\varphi) \to (Z,\eta)$ are factor maps
then there is a conjugacy $\gamma: (Z,\eta) \to (Z,\eta)$
such that $\beta_{1}= \gamma \circ \beta_{2}$.
\end{proposition}

\begin{proof}
Let $\beta_{1}, \beta_{2} : (X,\varphi) \to (Z,\eta)$ be two factor maps.
Let $(Y,\psi)$ be the unique odometer associated with $(X,\varphi)$ as in
Theorem~\ref{thm_odo}, and let $\alpha:  (X,\varphi) \to (Y,\psi)$
be as in that theorem. By Theorem~\ref{thm_odo}, there are
factor maps
$\gamma_{1}, \gamma_{2}: (Y,\psi) \to (Z,\eta)$
such that $\beta_{i}=\gamma_{i}\circ \alpha$ for $i=1,2$.
If we show that there is a conjugacy $\gamma: (Z,\eta) \to (Z,\eta)$
such that $\gamma_{1}= \gamma \circ \gamma_{2}$,
then it will follow that $\beta_{1}= \gamma \circ \beta_{2}$.

Let $(Y,\psi)$ and $(Z,\eta)$ be of type
$(r_{n})_{n=1}^{\infty}$ and $(s_{n})_{n=1}^{\infty}$, respectively.
We may assume that $(Y,\psi)=\mathcal{V}(D)$ and $(Z,\eta)=\mathcal{V}(C)$ for some
properly ordered Bratteli diagrams 
$D=(W,R,\geq)$ and $C=(U,S,\geq)$ where
$W=\bigcup_{n=0}^{\infty}W_{n}$, $U=\bigcup_{n=0}^{\infty}U_{n}$,
$R=\bigcup_{n=1}^{\infty}R_{n}$,
 $S=\bigcup_{n=1}^{\infty}S_{n}$,
$|W_{n}|=|U_{n}|=1$ for all $n\geq 0$, and $|R_{n}|=r_{n}$ and
$|S_{n}|=s_{n}$ for all $n\geq 1$.
Since $(Z,\eta)$ is dynamically homogeneous (see the proof of Theorem~\ref{thm_odo}),
there are conjugacies
$\delta_{1},\delta_{2}: (Z,\eta) \to (Z,\eta)$ such that
$\delta_{i}(\gamma_{i}(y_{\min}))=z_{\min}$, $i=1,2$,
where $y_{\min}\in Y$ and $z_{\min}\in Z$ are the unique minimal paths.
Since the contravariant functor $\mathcal{V}$ is full (by Theorem~\ref{thrfuncv}),
there are ordered premorphisms
$f,g:C\to D$ such that $\mathcal{V}(f)=\delta_1 \circ \gamma_1$
and $\mathcal{V}(g)=\delta_2 \circ \gamma_2$.
We claim that $f$ is equivalent to $g$.

Write $f=(\bigcup_{n=0}^{\infty}F_{n}, (k_{n})_{n=0}^{\infty},\geq )$ and
$g=(\bigcup_{n=0}^{\infty}G_{n}, (m_{n})_{n=0}^{\infty},\geq )$. Fix $n\geq 1$.
By symmetry, we may assume that $m_n \geq k_n$.
Consider the following (a priori non-commutative) diagram:

 \[
\xymatrix{U_{0}\ar[r]^{S_{0,n}}\ar[d]_{F_{0}=G_{0}}
 &U_{n}\ar[rd]^-{G_{n}}\ar[d]_{F_{n}}  \ \ \  \\
 W_{0}\ar[r]_{R_{0},k_{n}}
 &W_{k_{n}}\ar[r]_{R_{k_{n},m_{n}}} &W_{m_{n}} \   .
 }
\]
Let us show that the triangle in this diagram commutes.
(The square clearly commutes.)
Since $f$ is a premorphism, we have
$\mathrm{M}(F_{n})\mathrm{M}(S_{0,n})=\mathrm{M}(R_{0,k_{n}})$.
Thus, $\mathrm{M}(F_{n})={r_{1}\cdots r_{k_{n}}}/{s_{1}\cdots s_{n}}$
(as a $1\times 1$ matrix). Similarly,
since $g$ is a premorphism, we get
$\mathrm{M}(G_{n})={r_{1}\cdots r_{m_{n}}}/{s_{1}\cdots s_{n}}$.
Hence,
\[
 \mathrm{M}(R_{k_{n},m_{n}}) \mathrm{M}(F_{n})=
 \frac{r_{1}\cdots r_{m_{n}}}{r_{1}\cdots r_{k_{n}}}\cdot \mathrm{M}(F_{n})=
 \frac{r_{1}\cdots r_{m_{n}}}{s_{1}\cdots s_{n}}=
\mathrm{M}(G_{n}).
\]
Using Definition~\ref{defeq2}, $f$ is equivalent to $g$.
Hence,
$\delta_1 \circ \gamma_1=\mathcal{V}(f)=\mathcal{V}(g)=\delta_2 \circ \gamma_2$.
Put $\gamma= \delta_{1}^{-1} \circ \delta_2$ which is a
conjugacy from $(Z,\eta)$ to itself.
Then $\gamma_{1}= \gamma \circ \gamma_{2}$,
and the proof is complete.
\end{proof}

We note that Theorem~\ref{thm_odo}, Corollary~\ref{cor_exi_odo}, and
Proposition~\ref{prop_odo_uni} hold also for
essentially
minimal totally disconnected dynamical systems (Definition~\ref{defcatds})
which are slightly more general than Cantor minimal systems (the same proofs work).

\begin{figure}
\begin{center}
\begin{tikzpicture}[scale=1.5]

\filldraw (12,10) circle [radius=0.1];
\filldraw (10,7) circle [radius=0.1];
\filldraw (12,7) circle [radius=0.1];
\filldraw (14,7) circle [radius=0.1];
\filldraw (10,4) circle [radius=0.1];
\filldraw (12,4) circle [radius=0.1];
\filldraw (14,4) circle [radius=0.1];
\filldraw (10,1) circle [radius=0.1];
\filldraw (12,1) circle [radius=0.1];
\filldraw (14,1) circle [radius=0.1];

\draw (11.85,9.9)--(9.95,7.12);
\draw (11.9,9.68)--(10.13,7.05);

\draw (11.9,9.68)--(11.9,7.12);
\draw (12.1,9.68)--(12.1,7.12);

\draw (12.1,9.68)--(13.87,7.05);
\draw (12.15,9.9)--(14.05,7.12);

\node at (10.85,8.6) {\tiny{1}};
\node at (11.15,8.4) {\tiny{5}};
\node at (10.96,8.51) {.};
\node at (11,8.48) {.};
\node at (11.04,8.45) {.};

\node at (11.83,8.3) {\tiny{1}};
\node at (12.17,8.3) {\tiny{9}};
\node at (11.95,8.3) {.};
\node at (12,8.3) {.};
\node at (12.05,8.3) {.};

\node at (13,8.2) {\tiny{1}};
\node at (13.28,8.43) {\tiny{13}};
\node at (13.08,8.28) {.};
\node at (13.12,8.32) {.};
\node at (13.16,8.36) {.};

\draw (10,6.87)--(10,4.14);
\draw (12,6.87)--(12,4.14);
\draw (13.98,6.87)--(13.98,4.14);
\draw (14.06,6.87)--(14.06,4.14);
\draw (11.97,6.85)--(10.08,4.1);
\draw (13.95,6.85)--(12.08,4.1);
\draw (12.02,6.85)--(13.96,4.14);
\draw (10.1,6.9)--(11.91,4.1);
\draw (10.15,6.95)--(13.94,4.14);

\draw (10,3.87)--(10,1.14);
\draw (12,3.87)--(12,1.14);
\draw (13.98,3.87)--(13.98,1.14);
\draw (14.06,3.87)--(14.06,1.14);
\draw (11.97,3.85)--(10.08,1.1);
\draw (13.95,3.85)--(12.08,1.1);
\draw (12.02,3.85)--(13.96,1.14);
\draw (10.1,3.9)--(11.91,1.1);
\draw (10.15,3.95)--(13.94,1.14);

\node at (9.92,4.35) {\tiny{1}};
\node at (10.15,4.35) {\tiny{2}};
\node at (11.78,4.4) {\tiny{1}};
\node at (11.93,4.37) {\tiny{2}};
\node at (12.15,4.35) {\tiny{3}};
\node at (13.53,4.52) {\tiny{1}};
\node at (13.77,4.52) {\tiny{2}};
\node at (13.91,4.52) {\tiny{3}};
\node at (14.15,4.52) {\tiny{4}};

\node at (9.92,1.35) {\tiny{1}};
\node at (10.15,1.35) {\tiny{2}};
\node at (11.79,1.4) {\tiny{1}};
\node at (11.93,1.37) {\tiny{2}};
\node at (12.15,1.35) {\tiny{3}};
\node at (13.53,1.52) {\tiny{1}};
\node at (13.77,1.52) {\tiny{2}};
\node at (13.91,1.52) {\tiny{3}};
\node at (14.15,1.52) {\tiny{4}};

\filldraw (6,10) circle [radius=0.1];
 \filldraw (5,7) circle [radius=0.1];
\filldraw (7,7) circle [radius=0.1];
 \filldraw (5,4) circle [radius=0.1];
\filldraw (7,4) circle [radius=0.1];
 \filldraw (5,1) circle [radius=0.1];
\filldraw (7,1) circle [radius=0.1];

\draw (5.91,9.92)--(4.96,7.12);
\draw (5.98,9.88)--(5.05,7.13);

\draw (6.07,9.91)--(6.97,7.12);

\draw (4.95,6.87)--(4.95,4.14);
\draw (4.95,3.87)--(4.95,1.14);
\draw (5.03,6.87)--(5.03,4.14);
\draw (5.03,3.87)--(5.03,1.14);

\draw (6.95,6.87)--(6.95,4.14);
\draw (6.95,3.87)--(6.95,1.14);
\draw (7.03,6.87)--(7.03,4.14);
\draw (7.03,3.87)--(7.03,1.14);

\draw (6.9,6.9)--(5.05,4.15);
\draw (6.9,3.9)--(5.05,1.15);
\draw (5.1,6.9)--(6.9,4.15);
\draw (5.1,3.9)--(6.9,1.15);


\node at (5.06,7.7) {\tiny{1}};
\node at (5.32,7.7) {\tiny{2}};

\node at (7.11,4.39) {\tiny{3}};
\node at (6.89,4.39) {\tiny{2}};
\node at (6.66,4.39) {\tiny{1}};
\node at (5.12,4.7) {\tiny{2}};
\node at (4.89,4.7) {\tiny{1}};
\node at (5.33,4.7) {\tiny{3}};
\node at (7.11,1.39) {\tiny{3}};
\node at (6.89,1.39) {\tiny{2}};
\node at (6.66,1.39) {\tiny{1}};
\node at (5.12,1.7) {\tiny{2}};
\node at (4.89,1.7) {\tiny{1}};
\node at (5.33,1.7) {\tiny{3}};

\draw[->, thick] (6.1,10.1) [out=10, in=170] to (11.9,10.1);

\draw[->, thick] (5.15,7) [out=10, in=170] to (9.85,7);
\draw[->, thick] (5.15,7.05) [out=10, in=170] to (9.85,7.1);
\draw[->, thick] (7.15,6.95) [out=-10, in=190] to (9.85,6.88);

\draw[->, thick] (5.15,7.16) [out=20, in=156] to (11.85,7.24);
\draw[->, thick] (5.15,7.12) [out=20, in=158] to (11.85,7.12);
\draw[->, thick] (5.15,7.08) [out=20, in=160] to (11.85,7);
\draw[->, thick] (7.15,6.9) [out=-15, in=200] to (11.85,6.94);
\draw[->, thick] (7.15,6.86) [out=-15, in=205] to (11.85,6.85);
\draw[->, thick] (7.15,6.83) [out=-15, in=210] to (11.85,6.76);

\draw[->, thick] (5.15,7.32) [out=30, in=154] to (13.66,7.3);
\draw[->, thick] (5.15,7.28) [out=30, in=156] to (13.66,7.18);
\draw[->, thick] (5.15,7.24) [out=30, in=158] to (13.66,7.06);
\draw[->, thick] (5.15,7.2) [out=30, in=160] to (13.66,6.94);
\draw[->, thick] (7.15,6.8) [out=-30, in=200] to (13.66,6.9);
\draw[->, thick] (7.15,6.77) [out=-30, in=204] to (13.66,6.8);
\draw[->, thick] (7.15,6.74) [out=-30, in=208] to (13.66,6.7);
\draw[->, thick] (7.15,6.71) [out=-30, in=212] to (13.66,6.6);
\draw[->, thick] (7.15,6.68) [out=-30, in=216] to (13.66,6.5);

\draw[->, thick] (5.15,4) [out=10, in=170] to (9.85,4);
\draw[->, thick] (5.15,4.05) [out=10, in=170] to (9.85,4.1);
\draw[->, thick] (7.15,3.95) [out=-10, in=190] to (9.85,3.88);

\draw[->, thick] (5.15,4.16) [out=20, in=156] to (11.78,4.24);
\draw[->, thick] (5.15,4.12) [out=20, in=158] to (11.85,4.12);
\draw[->, thick] (5.15,4.08) [out=20, in=160] to (11.85,4);
\draw[->, thick] (7.15,3.9) [out=-15, in=200] to (11.85,3.94);
\draw[->, thick] (7.15,3.86) [out=-15, in=205] to (11.85,3.85);
\draw[->, thick] (7.15,3.83) [out=-15, in=210] to (11.85,3.76);

\draw[->, thick] (5.15,4.32) [out=30, in=154] to (13.66,4.3);
\draw[->, thick] (5.15,4.28) [out=30, in=156] to (13.66,4.18);
\draw[->, thick] (5.15,4.24) [out=30, in=158] to (13.66,4.06);
\draw[->, thick] (5.15,4.2) [out=30, in=160] to (13.66,3.94);
\draw[->, thick] (7.15,3.8) [out=-30, in=200] to (13.66,3.9);
\draw[->, thick] (7.15,3.77) [out=-30, in=204] to (13.66,3.8);
\draw[->, thick] (7.15,3.74) [out=-30, in=208] to (13.66,3.7);
\draw[->, thick] (7.15,3.71) [out=-30, in=212] to (13.66,3.6);
\draw[->, thick] (7.15,3.68) [out=-30, in=216] to (13.66,3.5);

\draw[->, thick] (5.15,1) [out=10, in=170] to (9.85,1);
\draw[->, thick] (5.15,1.05) [out=10, in=170] to (9.85,1.1);
\draw[->, thick] (7.15,0.95) [out=-10, in=190] to (9.85,0.88);

\draw[->, thick] (5.15,1.16) [out=20, in=156] to (11.78,1.24);
\draw[->, thick] (5.15,1.12) [out=20, in=158] to (11.85,1.12);
\draw[->, thick] (5.15,1.08) [out=20, in=160] to (11.85,1);
\draw[->, thick] (7.15,0.9) [out=-15, in=200] to (11.85,0.94);
\draw[->, thick] (7.15,0.86) [out=-15, in=205] to (11.85,0.85);
\draw[->, thick] (7.15,0.83) [out=-15, in=210] to (11.85,0.76);

\draw[->, thick] (5.15,1.32) [out=30, in=154] to (13.66,1.3);
\draw[->, thick] (5.15,1.28) [out=30, in=156] to (13.66,1.18);
\draw[->, thick] (5.15,1.24) [out=30, in=158] to (13.66,1.06);
\draw[->, thick] (5.15,1.2) [out=30, in=160] to (13.66,0.94);
\draw[->, thick] (7.15,0.8) [out=-30, in=200] to (13.66,0.9);
\draw[->, thick] (7.15,0.77) [out=-30, in=204] to (13.66,0.8);
\draw[->, thick] (7.15,0.74) [out=-30, in=208] to (13.66,0.7);
\draw[->, thick] (7.15,0.71) [out=-30, in=212] to (13.66,0.6);
\draw[->, thick] (7.15,0.68) [out=-30, in=216] to (13.66,0.5);

\node at (9.6,7.22) {\tiny{1}};
\node at (9.6,6.96) {\tiny{2}};
\node at (9.6,6.75) {\tiny{3}};

\node at (11.5,7.47) {\tiny{1}};
\node at (11.5,7.27) {\tiny{2}};
\node at (11.5,7.1) {\tiny{4}};
\node at (11.5,6.9) {\tiny{3}};
\node at (11.5,6.72) {\tiny{5}};
\node at (11.5,6.5) {\tiny{6}};

\node at (13.28,7.57) {\tiny{1}};
\node at (13.28,7.38) {\tiny{2}};
\node at (13.28,7.25) {\tiny{4}};
\node at (13.28,7.1) {\tiny{7}};
\node at (13.28,6.85) {\tiny{3}};
\node at (13.38,6.7) {\tiny{5}};
\node at (13.28,6.54) {\tiny{6}};
\node at (13.38,6.43) {\tiny{8}};
\node at (13.38,6.21) {\tiny{9}};

\node at (9.6,4.22) {\tiny{1}};
\node at (9.6,3.96) {\tiny{2}};
\node at (9.6,3.75) {\tiny{3}};

\node at (11.5,4.45) {\tiny{1}};
\node at (11.5,4.27) {\tiny{2}};
\node at (11.5,4.1) {\tiny{4}};
\node at (11.5,3.9) {\tiny{3}};
\node at (11.5,3.72) {\tiny{5}};
\node at (11.5,3.5) {\tiny{6}};

\node at (13.2,4.57) {\tiny{1}};
\node at (13.28,4.38) {\tiny{2}};
\node at (13.28,4.25) {\tiny{4}};
\node at (13.28,4.1) {\tiny{7}};
\node at (13.28,3.85) {\tiny{3}};
\node at (13.38,3.7) {\tiny{5}};
\node at (13.28,3.54) {\tiny{6}};
\node at (13.38,3.43) {\tiny{8}};
\node at (13.38,3.21) {\tiny{9}};

\node at (9.6,1.22) {\tiny{1}};
\node at (9.6,0.96) {\tiny{2}};
\node at (9.6,0.75) {\tiny{3}};

\node at (11.5,1.45) {\tiny{1}};
\node at (11.5,1.27) {\tiny{2}};
\node at (11.5,1.1) {\tiny{4}};
\node at (11.5,0.9) {\tiny{3}};
\node at (11.5,0.72) {\tiny{5}};
\node at (11.5,0.5) {\tiny{6}};

\node at (13.2,1.57) {\tiny{1}};
\node at (13.28,1.38) {\tiny{2}};
\node at (13.28,1.25) {\tiny{4}};
\node at (13.28,1.1) {\tiny{7}};
\node at (13.28,0.85) {\tiny{3}};
\node at (13.38,0.7) {\tiny{5}};
\node at (13.28,0.54) {\tiny{6}};
\node at (13.38,0.43) {\tiny{8}};
\node at (13.38,0.21) {\tiny{9}};

\node[very thick] at (6,0.2) {\vdots};
\node[very thick] at (12,-0.5) {\vdots};
\node at (6,10.7) {$C=(W,S,\geq)$};
\node at (12,10.7) {$B'=(V',E',\geq')$};
\node at (9,10.9) {$g$};
\draw[->] (7,10.7) to (10.8,10.7);
\end{tikzpicture}
\end{center}
\caption{An ordered premorphism $g$ from $C$
to $B'$. See Example~\ref{exa_Chacon2}.}\label{fig_Chacon2}
\end{figure}

Next let us  give an
illustrative example of two ordered
premorphisms which are inverses of each other.
This shows that ordered premorphisms can also be
used to verify conjugacy between Cantor minimal systems.
This also gives an alternative proof of the fact that
the  diagrams in Figure~\ref{fig_Chacon} are equivalent
(as already mentioned
 in \cite[Section~4.2]{GJ00}.)

\begin{example}\label{exa_Chacon2}
Let $(X,\varphi)$ (the Chacon system) and $C=(W,S,\geq)$ be as in Example~\ref{exa_Chacon}.
The diagram $C$ is
drawn on the left in  Figure~\ref{fig_Chacon2}. Let
$B=(V,E,\geq)$ be the properly ordered Bratteli diagram drawn on the left in Figure~\ref{fig_Chacon}
and let
 $B'=(V',E',\geq')$ be the  telescoping  to the sequence $0,3,4,5\ldots$ of  $B$.
The diagram $B'$ is drawn on the right in Figure~\ref{fig_Chacon2}.
It can be checked easily that $g:C\to B'$ in Figure~\ref{fig_Chacon2} is an ordered premorphism,
i.e., the ordered commutativity required in Definition~\ref{defpreobd} holds.
Write $g=(G, (n)_{n=0}^{\infty},\geq)$. Then
the multiplicity matrices
are the following:
\[
\mathrm{M}(E_{1}')
=\left(
 \begin{smallmatrix}
 5  \\
 9 \\
 13
 \end{smallmatrix}
 \right), \
\mathrm{M}(E_{n}')
=\left(
 \begin{smallmatrix}
 1 & 1 & 0 \\
 1 & 1 & 1\\
 1 & 1 & 2
 \end{smallmatrix}
 \right), \
 \mathrm{M}(S_{1})
=\left(
 \begin{smallmatrix}
 2  \\
 1
 \end{smallmatrix}
 \right), \
\mathrm{M}(S_{n})
=\left(
 \begin{smallmatrix}
 2 & 1  \\
 1 & 2
 \end{smallmatrix}
 \right),\  \text{for}\ n\geq 2;
 \]
 \[
  \mathrm{M}(F_{0})=(1),\
 \mathrm{M}(F_{n})
=\left(
 \begin{smallmatrix}
 2 & 1  \\
 3 & 3 \\
 4 & 5
 \end{smallmatrix}
 \right),
 \  \text{for}\ n\geq 1.
\]
Note that $g$ can also be considered as an ordered premorphism from
$C$ to $B$, in which  case we have $g=(G, (g_{n})_{n=0}^{\infty},\geq)$
where $g_0=0$ and $g_n=n+2$ for $n\geq 1$.
Also, $f$ in Example~\ref{exa_Chacon} can considered as an ordered
premorphism from $B$ to $C$ and in this case we can write
$f=(F, (f_{n})_{n=0}^{\infty},\geq)$,
where $f_0=0$ and $f_n=n+1$ for $n\geq 1$.
With this in mind, the compositions $fg: C\to C$ and
$gf:B\to B$ make sense.
It is easy to check that $fg\sim \mathrm{id}_{C}$ using
 Definition~\ref{defeq2}  for equivalence of ordered premorphisms
in the second sense. Applying the functor $\mathcal{V}$ we get
$\mathcal{V}([g]) \circ \mathcal{V}([f])= \mathrm{id}_{X}$.
Note that $\mathcal{V}([f])$ is surjective
(since factor maps between minimal systems are surjective).
It follows that $\mathcal{V}([f])$ is a conjugacy
and that $gf\sim \mathrm{id}_{B}$. In particular,
$B$ and $C$ are equivalent.
A dynamical  argument for this  can be obtained by using the fact that
the Chacon system $(X,\varphi)$ is topologically prime---i.e.,
it has no non-trivial factor.
\end{example}

\section{Weak Orbit Equivalence and C*-Algebras}\label{secwoe}
In this section we give an equivalent condition in terms of C*-algebras
for  weak orbit equivalence.

Let $(X,\varphi)$ and $(Y,\psi)$ be Cantor minimal systems. Recall
from \cite{gw95} that these systems are
\emph{weakly orbit equivalent} if there exists a homeomorphism $\alpha$ in
$[\varphi]$ such that the system $(X,\alpha)$ admits $(Y,\psi)$ as a factor, and
there exists a homeomorphism $\beta$ in
$[\psi]$ such that  $(Y,\beta)$ admits $(X,\varphi)$ as a factor.
(Here, $[\varphi]$ denotes the full group of $(X,\varphi)$; see \cite{gw95}.)
Two simple dimension groups with order unit, $G$ and $H$, are called
\emph{weakly isomorphic} if there exist order and order unit preserving group homomorphisms
from $G$ into $H$ and from $H$ to $G$.

For a C*-algebra $A$ let us denote by $\mathrm{T}(A)$  the set of tracial
states on $A$. When
$\mathrm{T}(A)\neq \emptyset$, there is a natural pairing
$\rho_{A}: \mathrm{K}_{0}(A)\to \mathrm{Aff}(\mathrm{T}(A))$
defined by $\rho_{A}([p])(\tau)=\tau(p)$ for all $[p]\in \mathrm{K}_{0}(A)$
and $\tau \in \mathrm{T}(A)$.
In the next result we have used the notion of UCT class. We refer the reader to
\cite[Definition~2.4.5]{Ro02} for the definition
and details.

\begin{theorem}\label{thmwoec}
Let $(X,\varphi)$ and $(Y,\psi)$ be Cantor minimal systems and set
$\mathrm{C}(X)\rtimes_{\varphi}\mathbb{Z}=A$ and
$\mathrm{C}(Y)\rtimes_{\psi}\mathbb{Z}=B$.
The following statements are equivalent:
\begin{enumerate}

\item\label{thmwoec_it1}
$(X,\varphi)$ and $(Y,\psi)$ are weakly orbit equivalent;

\item\label{thmwoec_it2}
there exists a positive homomorphism
from $\rho_{A}(\mathrm{K}_{0}(A))$ to $\rho_{B}(\mathrm{K}_{0}(B))$, mapping
 $\rho_{A}([1_{A}])$ to $\rho_{B}([1_{B}])$, and  one from
$\rho_{B}(\mathrm{K}_{0}(B))$ to $\rho_{A}(\mathrm{K}_{0}(A))$, mapping
$\rho_{B}([1_{B}])$ to $\rho_{A}([1_{A}])$;

\item\label{thmwoec_it3}
there are simple unital $A\mathbb{T}$ algebras  $C, D$ of real rank zero with $\mathrm{K}_{1}$ equal
to $\mathbb{Z}$ and $\mathrm{K}_{0}$ not equal to $\mathbb{Z}$ which are tracially equivalent to
$A$ and  $B$, respectively,
and there are unital  $*$-homomorphisms
from $C$ to $D$ and from $D$ to $C$;

\item\label{thmwoec_it4}
there are  separable simple unital C*-algebras $C, D$ with tracial rank zero which are tracially equivalent to
$A$ and  $B$, respectively,
and there are unital  $*$-homomorphisms
from $C$ to $D$ and from $D$ to $C$.
\end{enumerate}
Moreover,
in \eqref{thmwoec_it3} (and in \eqref{thmwoec_it4}, if we further assume that $C,D$ are
in the UCT class)
 we can replace the existence of  $*$-homomorphisms with the existence
of positive homomorphisms $\alpha : \mathrm{K}_{0}(C)\to \mathrm{K}_{0}(D)$ and
$\beta : \mathrm{K}_{0}(D)\to \mathrm{K}_{0}(C)$ such that
$\alpha([1_A])=[1_{B}]$ and $\beta([1_B])=[1_{A}]$ and that $\alpha,\beta$ preserve
the infinitesimal subgroups.
Also, we can choose $C, D$ in \eqref{thmwoec_it3} in such a way that the infinitesimal subgroups of
$\mathrm{K}_{0}(C)$ and $\mathrm{K}_{0}(D)$ are trivial.
\end{theorem}

\begin{proof}
First note that for any unital exact C*-algebra $A$ we have $\ker \rho_{A}=\mathrm{Inf}(\mathrm{K}_{0}(A))$, where
$\rho_{A}: \mathrm{K}_{0}(A)\to \mathrm{Aff}(T(A))$ is the natural pairing.

 \eqref{thmwoec_it1}$\Leftrightarrow$\eqref{thmwoec_it2}:
 This  follows from
 \cite[Theorem~2.3]{gw95}   and the relation between the
 K-theory of a Cantor minimal system and of the associated crossed product. In fact, let $A$ be as in the statement.
Then $ \mathrm{K}_{0}(A)/\mathrm{Inf}(\mathrm{K}_{0}(A))\cong \rho_{A}(\mathrm{K}_{0}(A))$ as
dimension groups with order unit, where the latter group is considered with the positive cone
$\rho_{A}(\mathrm{K}_{0}(A)^{+})$ and order unit $\rho_{A}([1_{A}])$.
On the other hand, $\mathrm{K}^{0}(X,\varphi)\cong \mathrm{K}_{0}(A)$ as dimension groups with order unit.
Thus,
\[
\rho_{A}(\mathrm{K}_{0}(A))\cong
\dfrac{\mathrm{K}^{0}(X,\varphi)}{\mathrm{Inf}\left(\mathrm{K}^{0}(X,\varphi)\right)}
\]
as
dimension groups with order unit. An analogous result holds for $B$. Now,
\cite[Theorem~2.3]{gw95}
implies that \eqref{thmwoec_it1} and \eqref{thmwoec_it2} are equivalent.

 \eqref{thmwoec_it1}$\Leftrightarrow$\eqref{thmwoec_it3}:
 There is a Cantor minimal system $(Z,\phi)$ such
 that
 \[
\mathrm{K}^{0}(Z,\phi)\cong \dfrac{\mathrm{K}^{0}(X,\varphi)}{\mathrm{Inf}\left(\mathrm{K}^{0}(X,\varphi)\right)},
\]
 as  dimension groups with order unit
(see \cite{pu89, hps92, gps95}). We may assume that $Z=X$. Indeed, let
 $h:X\to Z$ be a homeomorphism. Then $T=h^{-1}\phi h$ is a homeomorphism of $X$ and
 $h:(X,T)\to (Z,\phi)$ is a conjugacy. So
 $\mathrm{K}^{0}(X,T)\cong \mathrm{K}^{0}(Z,\phi)\cong
 \mathrm{K}^{0}(X,\varphi)/\mathrm{Inf}\left(\mathrm{K}^{0}(X,\varphi)\right)$ as
  dimension groups with order unit.
  Note that $\mathrm{Inf}\left(\mathrm{K}^{0}(X,T)\right)=0$. Then by
  \cite[Theorem~2.3]{gw95}, the systems
  $(X,\varphi)$ and $(X,T)$ are weakly orbit equivalent.
  Set $C=\mathrm{C}(X)\rtimes_{T}\mathbb{Z}$. Then by \cite[Theorem~4.2]{li05},
  $A$ and $C$ are tracially equivalent.
Similarly, there is a minimal homeomorphism $S$ of $Y$ such that
\[
\mathrm{K}^{0}(Y,S)\cong \dfrac{\mathrm{K}^{0}(Y,\psi)}{\mathrm{Inf}\left(\mathrm{K}^{0}(Y,\psi)\right)}.
\]
  Set $D=\mathrm{C}(Y)\rtimes_{S}\mathbb{Z}$. Thus, $B$ and $D$ are tracially equivalent. Note that
  $C$ and $D$ are simple unital $A\mathbb{T}$ algebras  of real rank zero with $\mathrm{K}_{1}$ equal
to $\mathbb{Z}$  and $\mathrm{K}_{0}$ not equal to $\mathbb{Z}$.
Since $(X,T)$ and $(Y,S)$ are weakly orbit equivalent, by
\cite[Theorem~2.3]{gw95} there exist positive
unital homomorphisms (i.e., morphisms in the category $\mathbf{DG}_{1}$)
$\alpha: \mathrm{K}_{0}(C)\to  \mathrm{K}_{0}(D)$ and
$\beta : \mathrm{K}_{0}(D)\to  \mathrm{K}_{0}(C)$. Note that $C$ and $D$ are
TAF~algebras and so
by \cite{da04}, there are unital $*$-homomorphisms $f: C\to D$ and $g:D\to C$ such that
$\mathrm{K}_{0}(f)=\alpha$ and $\mathrm{K}_{0}(g)=\beta$.

 \eqref{thmwoec_it3}$\Leftrightarrow$\eqref{thmwoec_it4}:
This follows from the fact that if $C$ is a simple unital $A\mathbb{T}$~algebra
of real rank zero with $\mathrm{K}_{1}$ equal
to $\mathbb{Z}$ and $\mathrm{K}_{0}$ not equal to $\mathbb{Z}$ then $C$ is a TAF~algebra.
In fact, by \cite[Theorem~1.15]{gps95}, there is a Cantor minimal system
$(Z,\phi)$ such that $C\cong \mathrm{C}(Z)\rtimes_{\phi}\mathbb{Z}$. By \cite{li05}, such an algebra
is a TAF~algebra.

 \eqref{thmwoec_it4}$\Leftrightarrow$\eqref{thmwoec_it2}:
Since $A$ is tracially equivalent to $C$, by \cite[Theorem~3.4]{li05} there is an order isomorphism
from
$\rho_{A}(\mathrm{K}_{0}(A))$ onto $\rho_{C}(\mathrm{K}_{0}(C))$ which maps
$\rho_{A}([1_{A}])$ to $\rho_{C}([1_{C}])$. Similarly, there is an order isomorphism
from
$\rho_{B}(\mathrm{K}_{0}(B))$ onto $\rho_{D}(\mathrm{K}_{0}(D))$ which maps
$\rho_{B}([1_{B}])$ to $\rho_{D}([1_{D}])$.
Now
let $f: C\to D$ and $g:D\to C$ be unital $*$-homomorphisms
as in \eqref{thmwoec_it4}. Then we get ordered group homomorphisms
$\mathrm{K}_{0}(f): \mathrm{K}_{0}(C)\to \mathrm{K}_{0}(D)$ and
$\mathrm{K}_{0}(g): \mathrm{K}_{0}(D)\to \mathrm{K}_{0}(C)$ which induce
ordered group homomorphisms from
$\rho_{C}(\mathrm{K}_{0}(C))$ to $\rho_{D}(\mathrm{K}_{0}(D))$ mapping
$\rho_{C}([1_{C}])$ to $\rho_{D}([1_{D}])$
and from
$\rho_{D}(\mathrm{K}_{0}(D))$ to $\rho_{C}(\mathrm{K}_{0}(C))$ mapping
$\rho_{D}([1_{D}])$ to $\rho_{C}([1_{C}])$.
By composing the appropriate maps we obtain
ordered group homomorphisms
from $\rho_{A}(\mathrm{K}_{0}(A))$ to $\rho_{B}(\mathrm{K}_{0}(B))$ mapping
 $\rho_{A}([1_{A}])$ to $\rho_{B}([1_{B}])$, and  from
$\rho_{B}(\mathrm{K}_{0}(B))$ to $\rho_{A}(\mathrm{K}_{0}(A))$ mapping
$\rho_{B}([1_{B}])$ to $\rho_{A}([1_{A}])$. Thus \eqref{thmwoec_it2} holds.

Observe that in \eqref{thmwoec_it3} and \eqref{thmwoec_it4} we may replace the
existence of unital $*$-homomorphisms with (unital)
maps between the $\mathrm{K}_0$-groups. This is  because the C*-algebras in question are separable simple unital
TAF~algebras in the UCT class and (by \cite{da04}) one can lift  unital positive homomorphisms between
 the $\mathrm{K}_0$-groups to unital $*$-homomorphisms
between the corresponding C*-algebras.
\end{proof}

\section*{Acknowledgements}
The third author thanks Maryam Hosseini for helpful discussions and for drawing
our attention to the article \cite{su11}.
The authors thank
Huaxin Lin and  Ian Putnam for helpful comments. The first and third authors were in part supported by a grant from IPM (No. 96430215 \& 94470072).
The research of the second author was supported by the Natural Sciences
and Engineering Research Council of Canada, and by the Fields Institute.
The research of the second author was supported by the Natural Sciences
and Engineering Research Council of Canada.

\end{document}